\documentclass[preprint,1p,fleqn]{elsarticle}
\usepackage{amsfonts,amsmath,amsthm,amssymb}
\usepackage{verbatim}
\usepackage{color}
\usepackage{moreverb}
\usepackage{slashbox}
\usepackage{caption}
\usepackage{multirow}

\usepackage{algorithm}
\usepackage{algpseudocode}
\usepackage{fixltx2e}

\usepackage{subfigure}
\usepackage{caption}
\usepackage{wrapfig}
\usepackage{epsfig}
\usepackage{subfig}

\def\NN{\hbox{I\kern-.2em\hbox{N}}}
\def\RR{{\mathop{{\rm I}\kern-.2em{\rm R}}\nolimits}}

\def\KK{{\bf K}}

\def\f{{\bf F}}
\def\g{{\bf g}}
\def\n{{\bf n}}
\def\x{{\bf x}}
\def\y{{\bf y}}
\def\w{{\bf w}}

\def\lam{{\lambda}}
\def\bfax{\mbox{\boldmath$\alpha$}}
\def\bfbx{\mbox{\boldmath$\beta$}}
\def\bfeta{\mbox{\boldmath$\eta$}}
\def\bmu{\mbox{\boldmath$\mu$}}
\def\btau{\mbox{\boldmath$\tau$}}
\def\bphi{\mbox{\boldmath$\phi$}}

\definecolor{orange}{rgb}{1,0.5,0}

\newcommand{\vertiii}[1]{{\left\vert\kern-0.25ex\left\vert\kern-0.25ex\left\vert #1
    \right\vert\kern-0.25ex\right\vert\kern-0.25ex\right\vert}}

\newcommand{\be}{\begin{equation}}
\newcommand{\ee}{\end{equation}}
\newcommand{\ba}{\begin{eqnarray}}
\newcommand{\ea}{\end{eqnarray}}
\newcommand{\supp}{\mathop{\mathrm{supp}}}

\newtheorem{rmk}{Remark}
\newtheorem{prn}{Proposition}

\begin{document}

\begin{frontmatter}

\title{An adaptive IGA-BEM with hierarchical B-splines based on quasi--interpolation quadrature schemes}


\author[label1]{Antonella Falini}
\address[label1]{INdAM c/o Department of Mathematics and Computer Science, University of
Firenze,\\
Viale Morgagni 67, Firenze, Italy }
\ead{antonella.falini@unisi.it}
\author[label2]{Carlotta Giannelli}
\address[label2]{Department of Mathematics and Computer Science, University of
Firenze,\\
Viale Morgagni 67, Firenze, Italy }
\ead{carlotta.giannelli@unifi.it}
\author[label1]{Tadej Kandu\v{c}}
\ead{tadej.kanduc@unifi.it}
%
\author[label3]{Maria Lucia Sampoli}
\ead{marialucia.sampoli@unisi.it}
\address[label3]{Department of Information Engineering and Mathematics, University of
Siena,\\
Via Roma 56, Siena, Italy}
\author[label2]{Alessandra Sestini}
\ead{alessandra.sestini@unifi.it}

\begin{abstract}
 
The isogeometric formulation of Boundary Element Method (BEM) is investigated within the adaptivity framework. Suitable weighted quadrature rules to evaluate integrals appearing in the Galerkin BEM formulation of 2D Laplace model problems are introduced.  The new quadrature schemes are based on a spline quasi-interpolant (QI) operator and properly framed in the hierarchical setting. The local nature of the QI perfectly fits with hierarchical spline constructions and leads to an efficient and accurate numerical scheme. An automatic adaptive refinement strategy is driven by a residual based error estimator.  Numerical examples show that the optimal convergence rate of the BEM solution  is recovered by the proposed adaptive method.

\end{abstract}


\end{frontmatter}

\section{Introduction}\label{sec:intro}

Boundary Element Methods (BEMs) are methods studied since the mid 1980s for the numerical solution of those Boundary Value Problems (BVPs), which can be transformed into Boundary Integral Equations (BIEs), see, e.g., \cite{BEMbook} for a recent overview. A common reference example is the Laplacian differential operator, but the theory can be extended also to more general partial differential equations, like the Helmholtz equation and the Stokes equations, which have applications in acoustics and fluid dynamics, respectively.

Boundary element methods have two main advantages: the dimension reduction of the computational domain and the simplicity for treating external problems. As a major drawback, the resulting integrals can be singular and therefore
robust and accurate quadrature formulas are necessary for their numerical computation. The solution is then obtained by collocation or Galerkin procedures.

The advent of \emph{Isogeometric Analysis} (IgA) \cite{ref2,LibroHughes} has brought a renewed interest in BEMs. In the IgA approach, a tight relation between the geometry of the domain and the representation of the approximate solution of the differential problem is established. In particular, IgA relies on a spline description of the domain, which is standard in Computer Aided Design (CAD), and on the usage of (possibly generalized and refined) analogous   spaces for the discretization of the differential problem.

In order to reduce complexity and gain efficiency, the use of collocation \cite{TauRodHug} or mixed collocation is the most common solution, see, e.g., \cite{Ginnis,  Simp2012}.
Indeed the Galerkin method has been generally avoided because it requires a double integration process, which appear difficult to evaluate efficiently. There are however applications, for example in crack propagation
problems, elasticity, elastodynamics, etc., where the use of a
Galerkin method may give some important advantages. Papers  using the IgA Galerkin BEM approach can be found in the literature, dealing with problems in acoustics \cite{Coox, SimpsonScott} or  flows \cite{politis, Joneidi}. BEM formulation has been used also to construct computational domains for Galerkin-IgA \cite{Falini}.
Recently, the IgA paradigm has been combined for the first time to the Symmetric Galerkin Boundary Element Method (IgA-SGBEM)\cite{ADSS1, ADSS2, Nguyen16}, which has revealed to be very effective among BEM schemes. Moreover, the full potential of B-splines over the more common Lagrangian basis has been recently exploited in \cite{ACDS3}.

When dealing with problems characterized by solutions with sharp features, adaptivity is a key ingredient to efficiently solve them; it requires suitable error estimations, as well as efficient local refinement procedures. While adaptive BEM have been widely studied in the literature, see, e.g., \cite{feischl2015arcme} for a recent review, the theory of adaptivity for isogeometric boundary element methods is still at a preliminary stage. A posteriori error analysis and refinement algorithms in the 2D setting have been presented in \cite{feischl2015reliable,feischl2016}. In \cite{feischl2017}, the optimal convergence of adaptive IgA-BEM for weakly singular equations was also proven. In all these studies, the adaptive scheme relies on the locally refinable nature of classical univariate B-splines, the standard spline basis adopted in CAD and IgA. In higher dimensions, however, the tensor-product B-spline structure does not provide local refinement capabilities and alternative spline spaces need to be considered.

One of the prominent approaches in the design and analysis of adaptive isogeometric methods exploits the multilevel structure of hierarchical B-splines, see, e.g., \cite{vuong2011}. A hierarchical B-spline space is constructed from a nested sequence of B-spline spaces, defined on different levels of resolution and on 
strictly localized parts of the domain. Non-uniform mesh configurations can be considered at different hierarchical levels. However, the adaptive nature of the spline hierarchy is usually considered on nested sequences of dyadically refined knots to simplify and speed up computations, while simultaneously providing an efficient adaptive framework. For the same reasons, the uniform configuration of hierarchical spline spaces based on a dyadic refinement is attractive also for a numerical treatment of 2D problems in the adaptive BEM.
\medskip

In this paper we present an adaptive IgA-BEM with hierarchical B-splines based on local quadrature schemes to solve 2D Laplace problems. The uniform structure of B-splines at different levels can be properly exploited in the assembly of the discretization matrices. In particular, the new quadrature schemes based on spline quasi--interpolation 
are investigated in this paper in the context of boundary integral equations and properly framed in the hierarchical setting.

The term {\it quasi--interpolation} (\emph{QI}) denotes a general approach to construct efficient local approximants to a given set of data or a given function; see \cite{LS75} for a general introduction to spline quasi--interpolation.
The quadrature rules adopted here are based on a quasi--interpolation operator firstly introduced in \cite{MSbit09} and applied to non singular numerical integration in \cite{MS12}. Then, the same QI-based idea was adopted in \cite{CFSS18} to develop efficient and competitive quadratures for singular integrals appearing in the IgA-BEM context.   The rule is effective also for nearly singular integrals.
Although the use of spline quasi--interpolation for numerical integration was already studied in several papers \cite{rabinowitz1990numerical, dagnino1993numerical, dagnino1997product, demichelis1996quasi}, its introduction in the IgA context is a novelty. In particular, in this work, the QI formulas derived in \cite{MS12} and \cite{CFSS18} are integrated in a Galerkin BEM model and suitably framed into a hierarchical adaptive scheme. In fact, the Galerkin Boundary Element Method is combined with an automatic adaptive refinement strategy driven by a residual based error estimator.  Numerical examples show the optimal convergence rates achieved by the hierarchical isogeometric scheme.


The structure of the paper is as follows. Section~\ref{sec:sgbem} presents the integral formulation of the model problem. Section~\ref{sec:bsplines} introduces B-spline representations of the domain boundary, commonly used in computer aided design systems. Hierarchical spline constructions and the definition of the isogeometric discretization are reviewed in Section~\ref{sub:hier} and \ref{sec:iga}, respectively. The adaptive scheme is summarized in Section~\ref{sec:adapt}. Section~\ref{sec:quad} introduces the quadrature formulas based on spline quasi--interpolation. In Section~\ref{sec:num} the developed  model is applied to an exterior and two interior 2D Laplace problems, all  suited for adaptivity. Finally, Section~\ref{sec:conc} concludes the paper.

\section{Integral formulation of the problem}\label{sec:sgbem}
In this work we focus on 2D exterior and interior Laplace model problems on
planar domains, assuming boundary Cauchy data of Dirichlet type. Two different geometries are considered: unbounded domains external to an open arc, and bounded simply connected domains.

In both cases the boundary of the domain $\Omega$ is a planar bounded curve $\Gamma$ without self-intersections. 
The boundary $\Gamma$ is described as the image of a regular invertible function $\f : [a,b]\rightarrow
\Gamma \subset \RR^2$, 
where $[a,b] \subset \RR$ is the parametric domain\footnote{Note that in case of a closed curve $\Gamma$, $\f : [a,b)\rightarrow\Gamma\subset\RR^2$, since $\f(a) =\f(b)$.}. 

When unbounded domains are considered, the differential problem is the following,
\begin{equation} \label{BVPsegment}
\left\{ \begin{array}{ll}
\Delta u=0& {\rm in}\; \Omega = \RR^2\setminus \Gamma,\\
u=u_{D}& {\rm on}\; \Gamma,
\end{array} \right.
\end{equation}
where the solution $u$ belongs to the Sobolev space $H^1(\Omega)$ and $u_D$ is the Dirichlet boundary datum, with $u_D\in H^{1/2}(\Gamma)$, the trace space of $H^1(\Omega)$. The BVP in (\ref{BVPsegment}) can model a variety of engineering problems including elasticity, fracture mechanics and acoustic (see for instance \cite{Coox,Nguyen16,Stephan84}) formulated on infinite domains.

In the second case, when $\Omega$ is a simply connected planar domain with a weakly Lipschitz boundary $\Gamma$, we deal with the interior Dirichlet problem
\begin{equation} \label{BVP_esse}
\left\{ \begin{array}{ll}
\Delta u=0& {\rm in}\; \Omega,\\
u=u_{D}& {\rm on}\; \Gamma,
\end{array} \right.
\end{equation}
where $u\in H^1(\Omega)$ and $u_D\in H^{1/2}(\Gamma)$ denotes again the given boundary datum.

The boundary element method uses the representation formula to evaluate the solution $u$ at any point $\x$ inside $\Omega$,
\begin{equation} \label{repformula}
u(\x)=-\frac{1}{2\,\pi}\int_{\Gamma}
U(\x,\y) \, \phi(\y) \, d\gamma_\y+\frac{1}{2 \pi} \, \int_{\Gamma}
\frac{\partial U}{\partial \n_y}(\x,\y) \, u_{D}(\y)\, d
\gamma_\y ,\quad \x \in \Omega,
\end{equation}
where 
$\bf n$ is the outward unit normal vector and $-\frac{1}{2\pi}\, U$ is the fundamental solution for the 2D Laplace operator, with
\begin{align*}\label{kernel}
U(\x\,,\,\y) := \log \|\x-\y\|_2.
\end{align*}
The function $\phi$ is the unknown and belongs to $H^{-1/2}(\Gamma)$, the dual space of $H^{1/2}(\Gamma)$, where the duality is defined with respect to the usual $L^2(\Gamma)$-scalar product.

In order to compute the missing boundary datum $\phi$, by applying a limiting process for $\x$ tending to $\Gamma$,  equation \eqref{repformula} allows us to derive a boundary integral equation (BIE) for both problems \eqref{BVPsegment} and \eqref{BVP_esse}.
In case of an infinite domain, an indirect approach \cite{chen} leads to the following BIE,
\be \label{prima_BIE}
-\frac{1}{2\,\pi}\int_{\Gamma}
U(\x,\y) \, \phi(\y)\, d\gamma_\y=u_{D}(\x),\quad \x \in \Gamma,
\ee
where the unknown $\phi\in H^{-1/2}(\Gamma)$ represents the jump of the flux of $u$.
In case of an interior problem, with a direct approach \cite{chen} we derive the BIE,

\be \label{seconda_BIE}
-\frac{1}{2\,\pi}\int_{\Gamma}
U(\x,\y) \, \phi(\y) \, d\gamma_\y=\frac{1}{2} \, u_{D}(\x)
-\frac{1}{2 \pi} \, \int_{\Gamma}
\frac{\partial U}{\partial \n_y}(\x,\y) \, u_{D}(\y)\, d
\gamma_\y ,\quad \x \in \Gamma,
\ee
with unknown $\phi\in H^{-1/2}(\Gamma)$ denoting the flux of $u$.
Both integrals equations \eqref{prima_BIE} and \eqref{seconda_BIE} are referred to as the Symm's integral equation
\be\label{Symmeq}
V\phi(\x) = f(\x),\quad \x\in\Gamma,
\ee
where $V:H^{-1/2}(\Gamma)\rightarrow H^{1/2}(\Gamma)$ is an elliptic isomorphism and corresponds to the operator
\begin{align*}
V\phi(\x) := - \frac{1}{2 \pi}\int_{\Gamma}U(\x,\y)\phi(\y)\,d\gamma_{\y}.
\end{align*}
The right hand side $f$ in \eqref{Symmeq} is given by $u_D$ in case of an indirect approach \eqref{prima_BIE}, or as the right hand side of \eqref{seconda_BIE} in case of a direct approach.

The variational formulation of \eqref{Symmeq} is \cite{Wendland2}:
\be \label{weakpb}
\text{\emph{given }}\ u_D \in H^{1/2}(\Gamma), \ \text{\emph{ find }}\ \phi \in H^{-1/2}(\Gamma)\; \ \text{\emph{ such that }}\
{\cal A}(\phi, \psi)={\cal F}(\psi),\quad \forall  \psi \in   H^{1/2}(\Gamma),
\ee
where the bilinear form ${\cal A}(\phi, \psi)$ and  right-hand side ${\cal F}(\psi)$ are defined as
\begin{align*} \label{bilf}
{\cal A}(\phi,\psi):= \int_{\Gamma} \psi(\x)\, V\phi(\x)\, d\gamma_\x,
\qquad
{\cal F}(\psi):= \int_{\Gamma} \psi(\x) \, f(\x)\, d\gamma_\x.
\end{align*}


\section{Isogeometric boundary element model}\label{sec:cagd}
In this section we summarize important properties of the model. First we give a quick overview of B-splines and its hierarchical extension. Then we reformulate equations from Section \ref{sec:sgbem} in a discrete form. In the last part we present important ingredients of the adaptive scheme.
\subsection{B-splines}\label{sec:bsplines}
A space of univariate B-splines of polynomial degree $d$ is uniquely defined by its \emph{knot} vector $ \mathbf T= \left\{t_1,\ldots,,t_{N+d +1}\right\}$. The knot vector defines $N$ spline basis elements. For the associated partition $\Theta$ on interval $[a,b]\subset\RR$ it holds
\begin{equation} \label{thetadef}
a = \theta_1 < \theta_2<\ldots <\theta_L = b,
\end{equation}
with $\theta_1=t_{d +1} = a$ and $\theta_L=t_{N+1} = b$. At every breakpoint $\theta_i$, for $i=2,\ldots,L-1$, the corresponding knots are repeated in the inner part of $\mathbf T$ with a multiplicity $m_i$, with $1\le m_i \le d +1$. The $d$ auxiliary knots on the left ($t_1,\ldots,t_{d}$) and on the right ($t_{N+2},\dots,t_{N+d +1}$) may be freely chosen, as long as they preserve the non-decreasing nature of the knot sequence, namely $t_i \le t_{i+1}$, for $i=1,\ldots N+d$. Note that the following spline dimension formula holds $N= d +1+{\sum_{i=2}^{L-1}}m_i$.

The \emph{B-spline} basis on $\mathbf T$ can be defined with the recursion formula \cite{deBoor01}
\begin{align*}\label{B_recurrence}
B_{i,0}(t) & := B_{i,0}^{(\mathbf T)}(t) := 
\left\{\begin{array}{ll}
1, & \quad\text{if}\; t_i\le t < t_{i+1},\\
0, & \quad \text{otherwise},
\end{array}\right. \\
B_{i,r}(t) & := \omega_{i,r}(t) B_{i,r-1}(t) + \left(1 - \omega_{i+1,r}(t)\right)B_{i+1,r-1}(t), \quad r=1,\ldots, d,
\end{align*}
where
\begin{align*}
\omega_{i,r} (t) := \left\{
\begin{array}{ll}
\frac{t-t_i}{t_{i+r}-t_i}, & \quad \text{if}\; t_i < t_{i+r},\\
0, & \quad \text{otherwise}.
\end{array}
\right.
\end{align*}
B-splines span a space of splines $S$,  whose smoothness is in general $C^{d-m_i}$ at the breakpoint $\theta_i$, for $i=2,\dots,L-1$.
B-splines have local support  --- the shape of $B_{i,r}$ depends only on its set of {\it active knots} $t_i,\ldots,t_{i+r+1}$ --- they are non-negative and (locally) linearly independent. The spline set $\{B_{i,d}\}_{i=1}^N$ also satisfies the partition of unity on $[a,b]$. The shape of B-splines defined on uniform partitions is shown in Figure~\ref{fig:bsplines1} for few low-degree cases.

\begin{figure}[t!]
\centering
\subfigure[Quadratic B-spline]{
\includegraphics[trim = 0.75cm 0.75cm 0.5cm 1cm, clip = true, height=1.25cm]{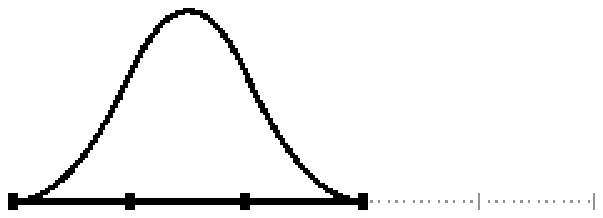}
}
\subfigure[Cubic B-spline]{
\includegraphics[trim = 0.75cm 0.75cm 0.5cm 1cm, clip = true, height=1.25cm]{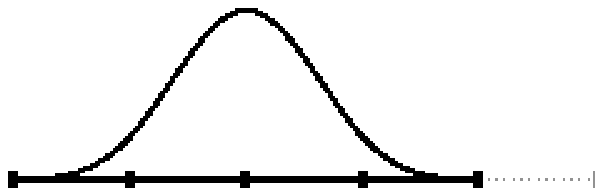}
}
\subfigure[Quartic B-spline]{
\includegraphics[trim = 0.75cm 0.75cm 0.5cm 1cm, clip = true, height=1.25cm]{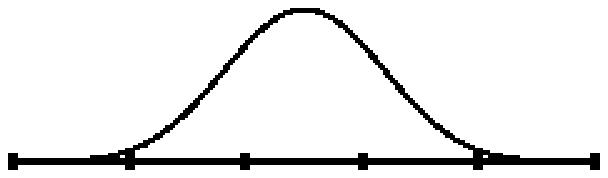}
}
\caption{B-splines of different degrees.}
\label{fig:bsplines1}
\end{figure}

Following the IgA paradigm, the boundary $\Gamma$ is parametrized by a parametric  B-spline curve $\f: [a,b] \rightarrow \RR^2$ written in the \emph{B-form}, 
\begin{align*} 
\f(\cdot) := \sum_{i=1}^N \mathbf{d}_i B_{i,d}(\cdot).
\end{align*}
The components of $\f$ belong to $S$ and $\left\{\mathbf{d}_i\right\}_{i=1,\ldots,N}$ is an ordered set of \emph{control points} in $\RR^2$. 
Thanks to the inherited properties of $S$, parametric B-spline curves are invariant to affine transformations and locally determined by only $d +1$ consecutive control points.
In addition, due to the properties of local support, partition of unity and non-negativity of the basis functions, any point on the B-spline curve lies in the convex hull of $d +1$ consecutive control points. This property is known as \emph{strong convex hull} and is of fundamental importance for geometric design applications.

To recover the interpolation of the first and the last control point, 
it is common to construct an \emph{open knot vector} by setting $t_1 = \ldots = t_{d} = a$ and $t_{N+2} = \ldots = t_{N+d+1} = b$. This is the standard choice to define open curves.

A periodic definition of the auxiliary knots is instead more convenient to represent closed curves, where $\f(a)=\f(b)$.
 In that case splines are thought to be periodic in a sense that 
a pair $\{B_{i,d}, B_{N-d+i,d}\}$, for $i=1,2,\dots,d$,
represents one shape function in the physical space. For the periodic compatibility it is sufficient that the $2d$ knot differences on the left are identical to the $2d$ ones on the right, $t_{i+1} - t_{i} = t_{i+N-d+1} - t_{i+N-d}$, for $i=1,\dots,2d$. Furthermore, the first $d$ control points of $\f$ need to coincide with the $d$ last ones.

\subsection{Hierarchical spline spaces}\label{sub:hier}

Adaptivity is of fundamental importance to obtain highly accurate solutions of numerical problems by increasing the number of degrees of freedom only in strictly localized regions.
In the IgA setting, the finite dimensional subspace used in the discretization of the differential problem is usually assumed coincident with $S$ or with a suitable enlargement of this space, in order to obtain sufficiently accurate approximations of the solution. In our approach we define such enlarged space through adaptive hierarchical $h$--{\it refinement} of fixed spline degree $d$.
We now review the construction of hierarchical B-spline spaces
needed for the development of the adaptive isogeometric boundary element method.	



In order to deal  with a sequence of uniform spline spaces at each refinement step, hierarchical  spaces, obtained by dyadic refinement, can be considered. Moreover, we assume a uniform partition $\Theta$ and  simple knots in $(a,b)$ in the associated knot vector $\mathbf T$. With this setting, a sequence of cardinal B-spline spaces, defined on uniform knot sequences at different level of details, can be constructed.

%
To properly define a spline hierarchy of this kind, we need to introduce a finite sequence of nested subdomains (not necessarily connected) of the parametric domain $[a, b]$,
\[
 {\hat{\Gamma}}^0 \supseteq {\hat{\Gamma}}^1 \supseteq \ldots \supseteq  {\hat{\Gamma}}^M, \qquad
 {\hat{\Gamma}}^M = \emptyset, \qquad  {\hat{\Gamma}}^0=[a,b].
\]
%
\emph{Mesh cells} of  $\hat{\Gamma}^0$ are determined by the partition $\Theta^0:=\Theta$, i.e., a cell is an element $(\theta_i,  \theta_{i+1})$ for $i=1,\dots,L-1$.
 The partition $\Theta^\ell$ of level $\ell$ is defined by dyadically refining the partition of the previous level, $\Theta^{\ell-1}$, for $\ell=1,\dots,M$, setting a simple multiplicity for every added knot. This means that any mesh cell $\hat{Q}$ of level $\ell$ is obtained by halving a cell of level $\ell-1$. Each $\hat{\Gamma}^\ell$ identifies the refinement region at level $\ell$ and it is the union of a certain number of cells defined on $\Theta^{\ell-1}$  \footnote{When using periodic splines on the extended knot vector we need to extend $\hat \Gamma^\ell$ also outside $[a,b]$ so that $a + s\in \hat \Gamma^\ell \iff b + s\in \hat \Gamma^\ell$ for $s\in \big[t_1^\ell-a, -t_1^\ell+a \big]$.}. The \emph{hierarchical mesh} $\hat{\cal Q}$ is defined as the collection of the active cells at different levels, namely 
$\hat{\cal Q}=\left\{\hat{Q}\in \Theta^\ell : \hat{Q}\subset \hat{\Gamma}^\ell \, \land \, \hat{Q}\not\subset\hat\Gamma^{\ell+1},\, \ell = 0,\ldots, M-1\right\}$.

A nested knot sequence of $\mathbf T^\ell$ is uniquely determined by the partition sequence of $\Theta^\ell$. Nested spline spaces are guaranteed by considering sets ${\cal B}^\ell = \big\{B_{1,d}^{(\mathbf T^\ell)}, B_{2,d}^{(\mathbf T^\ell)}, \dots ,B_{N_\ell,d}^{(\mathbf T^\ell)}\big\}$  of B-splines 
for $\ell=0,\dots,M-1$.
A basis for the hierarchical spline space defined on the hierarchical knot configuration is constructed by activating B-spline at finer levels on the refined subdomains. The linear independence of the basis can be guaranteed by eliminating coarser B-splines whose support is completely contained in the refined area. More precisely, 
we define the \emph{hierarchical basis} as
\[
{{\cal H}} := \left\{
B^{\cal H} \in {\cal B}^\ell : \supp(B^{\cal H}) \subseteq \hat{\Gamma}^{\ell} \wedge
\supp(B^{\cal H}) \not\subseteq \hat{\Gamma}^{\ell+1}
\right\},
\]
where each $B^{\cal H}\in \cal B^\ell$ corresponds to some $B_{i,d}^{(\mathbf T^\ell)}$ and $\supp(g)$ denotes the support of a function $g$ \footnote{\label{footsup} For a closed boundary $\Gamma$, a periodic definition of the basis $B_{i,d}^{(\mathbf T^\ell)}$ is adopted. A pair $\big\{ B_{i,d}^{(\mathbf T^\ell)}, B_{N_\ell-d+i,d}^{(\mathbf T^\ell)} \big\}$, for $i=1,\ldots,d$, is merged into one function. 
}.
It is also convenient to introduce a global numbering of the basis elements, 
$B_1^{\cal H}, B_2^{\cal H}, \dots, B_{N_{\cal H}}^{\cal H}$, where $N_{\cal H}$ is the cardinality of ${\cal H}$.



Quadratic B-splines and hierarchical B-splines defined on 3 refinement levels are shown in Figure~\ref{fig:bsplines3}. 
The properties of the hierarchical basis, as well as alternative basis constructions, were recently investigated, see, e.g., \cite{vuong2011,giannelli2012}.
The application of different kind of hierarchical spline refinement in isogeometric analysis is an active topic of research, see, e.g., \cite{vuong2011,giannelli2016,kanduc2017,schillinger2012}.

Note that under the previously mentioned assumptions on the boundary representation, the choice of performing dyadic refinements implies that $\cal H$ consists of functions which are all translates of $M$  dilations of a common reference B-spline defined on uniform knots. This means that every hierarchical B-spline of any level is simply the translate of a dilation of the reference cardinal B-spline. This assumption greatly simplifies and speeds up the implementation of quadrature rules (see Section~\ref{sec:quad}).
Furthermore, in the multivariate setting, the standard generalization of univariate B-splines through the tensor-product model prevents local refinement possibilities. The hierarchical B-spline basis instead can be used in any dimension as an effective adaptive spline construction.

\begin{figure}[t!]
\centering
\subfigure[B-splines of level 0 and ${\hat{\Gamma}}^0$]{
\includegraphics[trim = 1.75cm 0.65cm 1.25cm 0.5cm, clip = true, height=0.9cm]{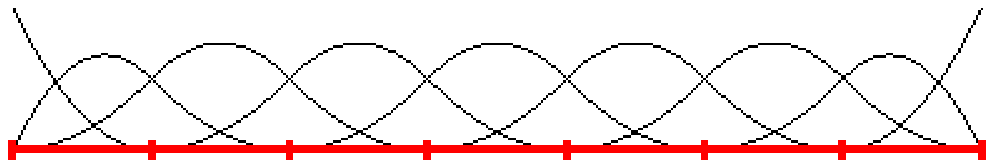}
}
\subfigure[Hierarchical B-splines, 1 level]{
\includegraphics[trim = 1.75cm 0.65cm 1.25cm 0.5cm, clip = true, height=0.9cm]{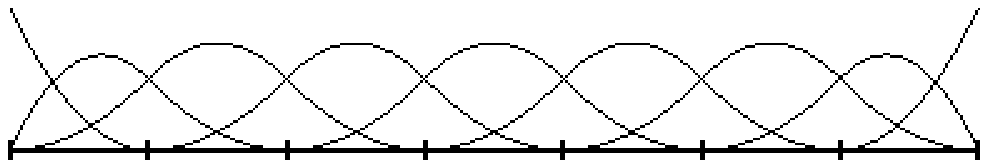}
}
\subfigure[B-splines of level 1 and ${\hat{\Gamma}}^1$]{
\includegraphics[trim = 1.75cm 0.65cm 1.25cm 0.5cm, clip = true, height=0.9cm]{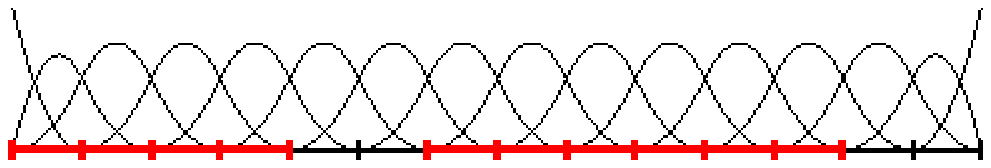}
}
\subfigure[Hierarchical B-splines, 2 levels]{
\includegraphics[trim = 1.75cm 0.65cm 1.25cm 0.5cm, clip = true, height=0.9cm]{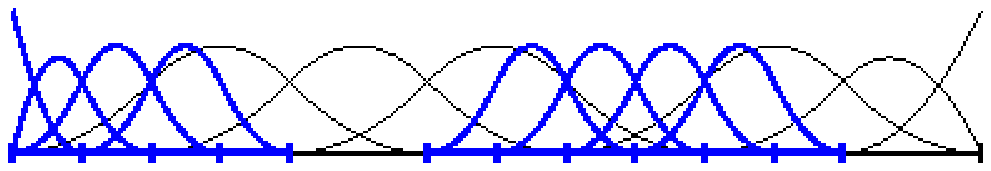}
}
\subfigure[B-splines of level 2 and ${\hat{\Gamma}}^2$]{
\includegraphics[trim = 1.75cm 0.65cm 1.25cm 0.5cm, clip = true, height=0.9cm]{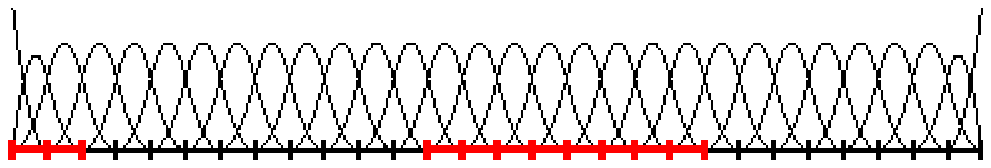}
}
\subfigure[Hierarchical B-splines, 3 levels]{
\includegraphics[trim = 1.75cm 0.65cm 1.25cm 0.5cm, clip = true, height=0.9cm]{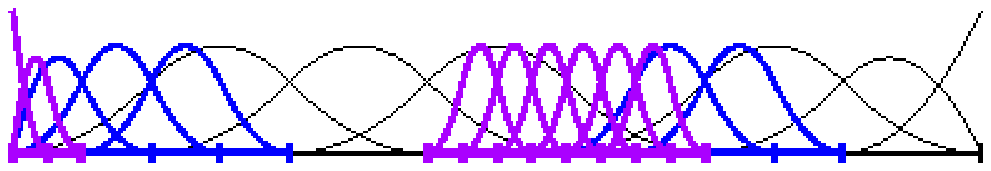}
}
\caption{B-splines (left) and hierarchical B-splines (right)
.}
\label{fig:bsplines3}
\end{figure}

\subsection{Isogeometric discretization}\label{sec:iga}

A discrete version of the given continuous variational problem (\ref{weakpb}) is obtained by approximating the infinite dimensional space $H^{-1/2}(\Gamma)$ with a finite dimensional subspace $S_{\cal H}$.
By adopting the Galerkin formulation the discrete problem reads as:
\be\label{weakpb2}
\text{\emph{given }}\ u_D \in H^{1/2}(\Gamma),\ \text{\emph{ find }}\ \phi_h \in S_{\cal H} \ \text{\emph{ such that }}\
{\cal A}(\phi_h, \psi_h)={\cal F}(\psi_h),\quad \forall \psi_h\in   S_{\cal H}.
\ee
The parameter $h$ in $\phi_h$ is related to the discretization step size of the subspace $S_{\cal H}$. In our setting, the approximation space $S_{\cal H}$ is generated by the lifted splines in ${\cal H}$, 
\begin{align*}
S_{\cal H} := 
\left \langle B_1^{\cal H} \circ \f^{-1}, B_2^{\cal H} \circ \f^{-1}, \dots, B_{N_{\cal H}}^{\cal H}\circ \f^{-1} \right \rangle.
\end{align*}
In the classical BEM setting, $S_{\cal H}$ is generated by functions obtained by lifting the $C^0$ piecewise polynomial Lagrangian basis in the parametric domain to the physical boundary.
Instead, in IgA the lifting is applied to the B-spline basis.


The applied Galerkin method leads to a linear system of  $N_{\cal H}$ equations and  $N_{\cal H}$ unknowns,
\begin{align} \label{sistlin}
 -\frac{1}{2\pi} A \,\bfax = \bfbx.
\end{align}
The unknown entries in the vector $\bfax = (\alpha_1,\dots,\alpha_{N_{\cal H}})^{T}$ are coefficients of the approximate solution of the problem \eqref{weakpb2}, $\phi_h(\x)  := \sum_{j=1}^{N_{\cal H}} \alpha_j (B_j^{\cal H}\circ \f^{-1})(\x)$.
{
The system matrix $A$ is symmetric and positive definite, and its entries $A^{(i,j)}$ are  the following 
double integrals: 
 \be \label{Gmatrix}
A^{(i,j)} := \int_{\Gamma} (B_i^{\cal H}\circ\f^{-1})(\x)
\int_{\Gamma} U(\x,\y)\ (B^{\cal H}_j\circ\f^{-1})(\y)\ d \gamma_\y \ d\gamma_\x.
\ee
The right-hand side vector $\bfbx \in \RR^{N_{\cal H}}$ depends on the given Cauchy data and on the problem at hand. Specifically, in the indirect approach it holds $\bfbx = \bfbx_1$, whereas in the direct approach we have $\bfbx = \frac 1 2 \bfbx_1 - \frac{1}{2 \pi} \bfbx_2$. The entries of $\bfbx_1$ and $\bfbx_2$ are
\be \label{Kt}
\bfbx_1^{(i)} := \int_{\Gamma}u_D(\x) (B^{\cal H}_i\circ\f^{-1})(\x)\, d \gamma_\x, \;
\bfbx_2^{(i)} := \int_{\Gamma} (B^{\cal H}_i\circ\f^{-1})(\x) \int_{\Gamma}
\frac{\partial U}{\partial \n_y}(\x,\y) \, u_D(\y)\, d
\gamma_\y d \gamma_\x. \ee

\subsection{Adaptive scheme}\label{sec:adapt}
The adaptive model iteratively computes an approximate B-spline solution on a hierarchical mesh for the given Laplace problem. The mesh is automatically refined at each step of the adaptive loop by taking into account the error estimator and the marking procedure  described below.

We consider a simple residual-based error estimator. Given the Symm's equation \eqref{Symmeq}, the residual $R_h$ is defined as
\begin{align*}
R_h(\x) := f(\x) - V\phi_h(\x), \quad \forall \phi_h\in H^{-1/2}(\Gamma).
\end{align*} 
Since $V$ is an elliptic isomorphism between $H^{-1/2}(\Gamma)$ and $H^{1/2}(\Gamma)$, and the right hand side $f$ belongs to $H^{1/2}(\Gamma)$, then also $R_h\in H^{1/2}(\Gamma)$. 
The norm in  $H^{1/2}$ can be defined using the following Sobolev-Slobodeckij norm:
\begin{equation}\label{DefSlobodeckij}
\|u\|^2_{H^{1/2}(\Gamma)} := \|u\|^2_{L^2(\Gamma)} + |u|^2_{H^{1/2}(\Gamma)}.
\end{equation}
The symbol $|\cdot|$ denotes the seminorm in $H^{1/2}$, given by
\begin{align*}
|u|_{H^{1/2}(\Gamma)}^2 := \int_{\Gamma}\int_{\Gamma}\frac{\|u(\x)-u(\y)\|^2}{\|\x-\y\|^2}\,d\gamma_\y\,d\gamma_\x.
\end{align*}
Let us assume that a partition of the boundary $\Gamma$ into connected components $\Gamma_i$ is given, such that ${\Gamma = \bigcup_{i} \Gamma_i}$ and for $i\neq j$ the intersection $\Gamma_i \cap \Gamma_j$ is either empty or a common point. Then it is easy to see from definition \eqref{DefSlobodeckij} that
\begin{equation}\label{nonlocal}
\|u\|^2_{H^{1/2}(\Gamma)} = \sum_{i}\|u\|^2_{H^{1/2}(\Gamma_i)} + \sum_{\substack{i,j \\
 i\neq j
} }\int_{\Gamma_i}\int_{\Gamma_j}\frac{\|u(\x)-u(\y)\|^2}{\|\x-\y\|^2}\,d\gamma_\y\,d\gamma_\x. 
\end{equation}
The double-integral definition of the semi norm in \eqref{nonlocal} prevents us to split the norm on $\Gamma$ into a sum of locally defined contributions, i.e., $\|u\|_\Gamma^2 \neq \sum_i \|u\|_{\Gamma_i}^2$. See for instance \cite{carstensen2001mathematical} for a demonstration of non-locality of fractional Sobolev spaces even for a more general condition: $\|u\|_\Gamma^2 \not\leq C\, \sum_i \|u\|_{\Gamma_i}^2$ for any $C>0$.

Following the construction from \cite{feischl2015reliable} we introduce the notion of overlapping \emph{patch domains} in order to define a local approach for the considered error estimator. 

Given a hierarchical mesh in the physical domain, ${\cal Q}= \{Q = \mathbf{F}(\hat{Q}) : \hat{Q}\in \hat{\cal Q}\}$, for every mesh cell $Q$ the patch $\omega(Q)$ collects its neighbouring elements,
\begin{align*}
\omega(Q):= \bigcup\{\overline{Q^\prime}\in\mathcal{Q}: \overline{Q^\prime}\cap \overline{Q}\neq\emptyset \}.
\qquad 
\end{align*}
Then, the residual-based error estimators $\eta_h(Q)$ and $\eta_h$ are constructed on every patch domain $\omega(Q)$ and for the whole mesh $\mathcal{Q}$, respectively,
\begin{align*}
\eta_h^2(Q) := |R_h|^2_{H^{1/2}(\omega(Q))},\qquad \eta_h^2 :=\sum_{Q\in {\cal Q}} \eta_h^2(Q).
\end{align*}
The a posteriori error analysis of the indicator $\eta_h$ was developed in \cite{feischl2015reliable}, where non-uniform (rational) B-splines were considered.
To read about other types of error estimators for BEM, see for instance \cite{carstensen2001mathematical} and references therein.


At any step $k$ of the adaptive loop, we solve the model problem with hierarchical spline spaces defined on the current hierarchical mesh. By computing the residual error estimator, we apply the \textit{D\"{o}rfler marking} \cite{dorfler1996convergent}, that determines the set ${\cal M \subset \cal Q}$ of cells marked for refinement so that, 
\begin{align*}
\theta\,\eta_h \leq \sum_{Q\in\mathcal{M}}\eta_h(Q),
\end{align*}
with respect to the marking parameter $\theta\in(0,1]$. The refined hierarchical mesh to be considered at the step $k+1$ of the adaptive loop is obtained by dyadically refining any marked element.

%
%
%


\section{Quadratures}\label{sec:quad}

In this section we describe our novel quadrature schemes for evaluating the integrals that appear in \eqref{sistlin} as the entries of the coefficient matrix $A$ and of the vector $\bfbx$ on the right--hand side of the linear system. 
The main idea of the schemes is to rewrite the integrand functions in terms of simpler functions that can be efficiently integrated.  Since the integration domain is always rectangular, any double integral can be easily split into two single ones. The quadrature rules do not limit the position of the nodes (a node can also coincide with a singular point of the kernel); furthermore the quadrature nodes for any single integral are always chosen uniformly spaced,  in order to speed up the construction of the quadrature weights. Some theoretical and experimental results related to the accuracy of the schemes are introduced at the end of the section.

Firstly, the integrals are rewritten in the parametric space. Secondly, the kernel splitting of $U$ introduced  in Section  \ref{Kernelsplit} allows us to separate the geometrical influence on the integrand function from the part related to the singularity of the kernel. The details on the splitting procedure for the case of open boundary curves were already described in \cite{ACDS3}. The splitting technique related to the closed curve geometries is new in this paper. Then, preliminaries of the utilized spline quasi--interpolation operator are summarized in Section \ref{prel}. The considered quasi--interpolating approach is a variant of the Hermite scheme introduced in \cite{MSbit09} to approximate a sufficiently regular function $g$, and successively applied in \cite{MS12} to generate an efficient quadrature rule for $\int_a^b g$.

In Sections \ref{Qregsplit} and \ref{Qsingsplit} the scheme is applied as a subroutine to evaluate the following two kinds of regular and weakly singular integrals in 1D:
\begin{equation} \label{intreg}
I_{B_i^{\cal H}} [g] := \int_{D_i} B_{i}^{\cal H}(s)\,  g(s) \ ds,
\end{equation}
 and
\begin{equation} \label{singint}
I_{ {\it w}_i^s} [g] :=\int_{D_i}\log\delta(s,t), B_i^{\cal H}(t)\,  g(t)\ dt. 
\end{equation}
Here we denote $D_i := \supp(B_i^{\cal H})$ and ${\it w}_i^s(\cdot) := \log\delta(s,\cdot)\, B_i^{\cal H}(\cdot)$, and we assume $g \in C(\overline{D}_i)$.
Integrals in \eqref{singint} are considered weakly singular if $s \in {D}_i$, nearly singular if $s \notin {D}_i$ but the distance between $s$ and $D_i$ is sufficiently small, and regular otherwise.

Quadrature techniques for the two integrals share two common basic ideas. First, $g$ is approximated on $D_i$ by a quasi--interpolant spline $\sigma_g$ of a degree $p$. The quasi--interpolant depends only on the values of $g$ at the spline breakpoints. Second, by applying the spline product algorithm in \cite{Morken91}, the product $B_i^{\cal H} \sigma_g$ is expressed as a linear combination of particular B-spline basis functions of degree $p+d$. Since the definite integral of any B-spline is well known, this strategy immediately generates an associated quadrature rule for \eqref{intreg}.  The details are reported in Section \ref{Qregsplit}. Singular integrals in \eqref{singint} are computed by utilizing a quadrature scheme that was developed and tested in \cite{CFSS18}. As shown in Section \ref{Qsingsplit}, it requires the preliminary computations of the \emph{modified moments} for B-splines, which were firstly introduced in \cite{ACDS3} for open boundary curves. Modified moments suitable for a closed geometry can be straightforwardly obtained from the original formulation, as demonstrated at the end of the subsection.
Important steps of the derived quadrature scheme are summarized in a compact algebraic form to evaluate the double integrals in the system matrix $A$.
Finally, in Section \ref{AP} the approximation power of the scheme is analyzed both theoretically and empirically.

\subsection{Splitting the kernel} \label{Kernelsplit}

By introducing coordinates $s,\,t \in [a, b]\subset \RR$ in the parametric domain,
\begin{align*} 
s := \f^{-1}(\x), \quad t := \f^{-1}(\y),
\end{align*}
and taking into account the local support of B-splines, the double integrals in \eqref{Gmatrix}  can be expressed in the parametric intervals as
 \be \label{final_1} I^{(i,j)}:=\int_{D_i} B_{i}^{\cal H}(s)\, J(s)
\int_{D_j}  U(\f(s),\f(t))\, \ B_{j}^{\cal H}(t)\, J(t)  \ dt \ ds,
\ee
where, $J$ is the {\it parametric speed} associated with $\Gamma,$
\begin{align*}
J(\cdot) := \|\f'(\cdot)\|_2.
\end{align*}

In order to separate the contribution of the geometry to the kernel $U$ from the part affected by the singular nature of the kernel, we can write
 \[
U(\f(s),\f(t)) = K_1(s,t) + K_2(s,t)\,,
\]
by defining
\begin{align*}
K_1(s,t) := \frac 1 2 \log \frac{\|\f(s) - \f(t)\|^2_2}{\delta^2(s,t)},\qquad
K_2(s,t) :=  \log \delta(s,t).
\end{align*}
The function $\delta$ needs to be carefully chosen so that the resulting splitting into $K_1$ and $K_2$ simplifies the corresponding integrals. The kernel $K_1$, which carries the geometry information, should be of the highest possible regularity ; specifically the definition of the kernel and its partial derivatives needs to be extendible to the region where $\f(s)=\f(t)$.
Moreover, integrals containing the kernel $K_2$ should have a simple enough computable expressions.
This motivates the definition of $\delta$ according to the type of the boundary domain $\Gamma$,
\begin{align*}
\delta(s,t) := \left\{
\begin{array}{ll}
|s - t|, & \quad \textrm{if $\Gamma$ is an open curve}, \\
|s - t|\, {|(s-t)^2 - \gamma^2|}\,{\gamma^{-2}}, & \quad \textrm{if $\Gamma$ is a closed curve},
\end{array}
\right.
\end{align*}
where $\gamma := b-a$.

The limit case for $\delta(s,t) = |s-t|$ was already proven in \cite{ACDS3},
$$\lim_{ t \to s}  K_1(s,t) = \log J(s)\,.$$
Hence $K_1$ is well defined on $[a,b]^2$, when the boundary is an open curve.
On the other hand, it can be checked analogously that when closed boundary curves are considered, the new definition of $K_1$ guarantees that the above limit holds as well. Furthermore, since $\f(a)=\f(b)$, the other two factors in $\delta$ are added so that the kernel $K_1$ is well defined on $[a,b]^2$ also when $(s-t) \rightarrow \pm (b-a)$.

The integral defined in \eqref{final_1} can then be evaluated as
$$I^{(i,j)}=I_{K_1}^{(i,j)} + I_{K_2}^{(i,j)}$$ with
\begin{align} \label{IK12}
I_{K_r}^{(i,j)} : = \int_{D_i} B_{i}^{\cal H}(s)\, J(s) \int_{D_j}
K_{r}(s,t)\, B_{j}^{\cal H}(t)\, J(t)\,dt \ ds, \qquad r=1,2.
\end{align}

\subsection{Quasi--interpolation quadrature scheme} \label{prel}

For integrals \eqref{intreg} and \eqref{singint} the quadrature formula utilizes quasi--interpolation spline techniques as subroutines. Without loss of generality let us assume in this subsection that the integration domain is $[0, 1]\subset\RR$. 
When $g$ is a regular function, the integral $\int_0^1 g$ is approximated by $\int_0^1 \sigma_g$. The quasi--interpolating spline $\sigma_g$ is described by its breakpoints on interval $[0,1]$ and by its locally defined spline control coefficients. In particular, we rely on the quasi--interpolation scheme introduced in \cite{MSbit09}, which uses splines of maximal smoothness and  defines the spline coefficients using only discrete information at the spline breakpoints.
Thus, the approximant $\sigma_g$ is defined in the spline space $S_{\btau}$ of degree $p$ and on the open knot vector $ \btau := \{\tau_{-p},\ldots,\tau_{n+p} \}$, with $ 0= \tau_{-p}=\ldots=\tau_0 < \cdots < \tau_n = \ldots=\tau_{n+p} = 1$:
\begin{align*}
S_{\btau} :=  \left \langle B_{-p,p} ^{(\btau)}, \dots,B_{n-1,p} ^{(\btau)} \right \rangle. %
\end{align*}
The breakpoints of $\btau$ define $n+1$ quadrature nodes. 
The B-form of the quasi--interpolant $\sigma_g$ reads
\begin{align*} 
\sigma_g :=\sum_{j=-p}^{n-1}\lam_j(g)\, B_{j,p}^{(\btau)},
\end{align*}
where  each  $\lam_j (g)$ is defined as a  suitable  linear combination of  a local subset of $g$ and possibly $g'$ values at the spline breakpoints.
For instance, for $p=2$, they can be obtained as
$$
\begin{array}{rl}
\lam_j(g) &=\frac{1}{2}\left( g(\tau_{j+1})+g(\tau_{j+2})\right)-  \frac{\tau_{j+1}-\tau_j}{4}\left( -g'(\tau_{j+1})+g'(\tau_{j+2})\right), \quad  j=-1,\ldots,n-2,\cr
\lam_{-2}(g) &=\,g(\tau_0), \qquad \lam_{n-1}(g)=g(\tau_n).
\end{array}
$$
This quasi--interpolation scheme has the optimal approximation order $p+1$ for $g\in C^{p+1}[0,1]$. 

To avoid problems, when derivatives of $g$ are not available or their computation is too expensive, we focus on a modified version of the scheme, where the derivative values of $g$ are automatically approximated by suitable finite difference formulas \cite{MS12}.
The later scheme ${\cal I}^Q[g]$ to approximate $\int_0^1 g$ reads
\begin{equation} \label{Q1}
{\cal I}^Q[g] \,:=\, \bphi^T\ \g,
\end{equation}
where $\g:= \left( g(\tau_0),\ldots, g(\tau_n) \right)^T$, and the weight vector $\bphi$ has the following structure,
\[
\bphi:= \hat{C}^{(p)} \ \KK^{(p)},
\]
with $\KK^{(p)}  \,:=\, (k_j) \in \RR^{n+p+1}$ and $k_j :=  \int_{\supp(B_{j,p}^{(\btau)})}  B_{j,p}^{(\btau)}(s) \ ds$.
We recall the definite integral of an arbitrary B-spline $B_{i,d}$ is
\begin{align}\label{eqn:integBspline}
 \int_{\supp(B_{i,d})}  B_{i,d}(s) \ ds = \frac{L_{\rm s}}{d+1},
\end{align}
where $L_{\rm s}$ is the size of $\supp(B_{i,d})$ and $d$ is the spline degree.
The other factor appearing in the definition of $\bphi$ is the matrix $\hat{C}^{(p)} \in \RR^{(n+1) \times (n+p+1)}$ which is banded with bandwidth depending on $p$. 

Referring to \cite{MS12} for the details, we highlight just two important properties of the rule in (\ref{Q1}). First, when $\btau$ is uniform, the matrix $\hat{C}^{(p)}$ greatly simplifies and hence so does $\bphi$. Second, we recall the scheme's convergence behavior. By denoting with $|\btau|$ the maximal distance between two consecutive 
breakpoints in $\btau$,  the quadrature error is bounded by $C\, \| D^{p+1} g \|_{L^{\infty}} |\btau|^{p+1}$ if $g \in C^{p+1}[0,1]$. Furthermore, for $p$ even and symmetric mesh on the integration interval, the order increases to  $O(|\btau|^{p+2})$.

\begin{rmk}\label{rmk:QI}
 The described quadrature \eqref{Q1} can be applied on integrals \eqref{intreg} by approximating the whole product $B_i^{\cal H} g$ with the quasi--interpolation spline. The drawback of this approach is that the accuracy of the integration scheme depends also on the regularity of $B_i^{\cal H}$. Hence, if $B_i^{\cal H}$ is locally less regular than $g$, the accuracy of the rule is reduced. Furthermore, the norm $\| D^{p+1} (B_i^{\cal H} g) \|_{L^{\infty}(\overline{D}_i)}$ in the estimation grows with smaller $D_i$. To overcome these limitations, a quadrature rule with a separate B-spline factor is introduced in the following subsection.
\end{rmk}


\subsection{Quadrature for regular integrals with a B-spline factor} \label{Qregsplit}

In this subsection we describe a specific quadrature rule to handle integrals of the type introduced in \eqref{intreg}. 
When switching to the parametric domain, such integrals 
appear in the system matrix in $I_{K_1}$ and in $I_{K_2}$ for regular integrands (see \eqref{IK12}), and also in the right-hand side vectors, that is in $\bfbx_1$ and in outer integrals of $\bfbx_2$ in \eqref{Kt}. 

Following the construction from the previous subsection, the function
$g$ in \eqref{intreg} is approximated by the QI spline approximant $\sigma_g$ on $D_i$. The spline lies in $S_{\btau^{(i)}}$, where $\btau^{(i)}$ is the open knot vector associated to a uniform partition of $D_i$ into $n$ subintervals. The product $ B_{i}^{\cal H}\,  \sigma_g$ is a spline of degree $p+d$ 
defined on $D_i$. Referring to \cite{Morken91} for the details, the product can be expressed in B-form in a new basis, 
$\{B^{(\btau_{\Pi})}_{k,p+d}\}_{k=1}^{P}$, which spans the product space $\Pi$, defined on the knot vector  $\btau_\Pi$. The dimension $P$ of the product space $\Pi$ depends on $n$, $d$ and $p$. The definite integral of any $B^{(\btau_\Pi)}_{k,p+d}$ can be easily computed from formula  \eqref{eqn:integBspline}. Thus, 
the quadrature rule ${\cal I}^Q_{B_{i}^{\cal H}}[g]$ for $I_{B_i^{\cal H}} [g] $ for the auxiliary function $g$ can be expressed as
\be \label{Q2}
{\cal I}^Q_{B_{i}^{\cal H}}[g] \,:=\, {\w^{(i)}}^T\ \g^{(i)},
\ee
with $ \g^{(i)} := \left( g(\tau_0^{(i)}),\ldots, g(\tau_n^{(i)}) \right)^T$ and the weight vector $\w^{(i)}$ defined as
\[
\w^{(i)} :=  \hat{C}^{(p)} \ G^{(p,d)}  \ \KK^{(d+p)}.
\]
The matrix $G^{(p,d)} \in \RR^{(n+p+1) \times P}$ is 
easily defined from the formulas in \cite{Morken91} to obtain a compact representation of all the coefficients of the product $ B_{i}^{\cal H}\,  \sigma_g,$ starting form the B-spline representation of $\sigma_g.$ 
Note that the weight vector $\w^{(i)}$ depends on the index $i$  only because the entries of $\KK^{(d+p)}$ depend on a scale factor given by the size of $D_i$.

The advantage of \eqref{Q2} is apparent when the auxiliary function $g$ is locally smoother than the B-spline factor and $d\leq p+1$, since in such case the rule in (\ref{Q1}) would not reach its maximal approximation power (see Remark~\ref{rmk:QI}). Moreover, the error of the quadrature \eqref{Q2} does not depend on the norm of the derivatives of $B_i^{\cal H}$, which can be arbitrarily large with the smaller size of $D_i$. The convergence properties of the scheme \eqref{Q2} is analyzed in Section~\ref{AP}.

As mentioned before, the introduced rule  \eqref{Q2} is employed in the assembly phase of our model for three kinds of non-singular integrals. 
In particular, the double integrals $I_{K_1}^{(i,j)}$ are approximated by the following scheme
\be \label{eqn:IK1}
I_{K_1}^{(i,j)} \approx {\w^{(i)}}^T\  J^{(i)} \
\left( \begin{array}{lll}
K_1(\tau^{(i)}_0,\tau^{(j)}_0) & \cdots &K_1(\tau^{(i)}_0,\tau^{(j)}_{n})\cr
\vdots & \vdots & \vdots \cr
K_1(\tau^{(i)}_n,\tau^{(j)}_0) & \cdots &K_1(\tau^{(i)}_n,\tau^{(j)}_n)
\end{array}
\right) J^{(j)} \w^{(j)},	 
\ee
where $J^{(k)}:=\mbox{diag}\left(J(\tau^{(k)}_0),\ldots,J(\tau^{(k)}_n) \right)$.
The same approach is applied also for the numerical computation of $I_{K_2}^{(i,j)}$, whenever it is regular.
The entries of the vector $\bfbx_1$ are obtained by setting $g(\cdot) = J(\cdot) (u_D \circ \f^{-1})(\cdot)$. A similar structure to \eqref{eqn:IK1} is applied to approximate also the entries of $\bfbx_2$, which appear in the formulation for interior problems.
 
\begin{rmk}
Entries in $\bfbx_2$ are regular integrals due to assumption $\f \in C^2[a,b]$; see \cite{ACDS3}, Section 3.2.
\end{rmk}

\subsection{Quadrature for singular integrals with a B-spline factor} \label{Qsingsplit}

To address weakly singular and nearly singular integrals of the type \eqref{singint} an extension of the rule \eqref{Q2} has been recently developed \cite{CFSS18} (namely {\tt procedure~2} in Section 5).
Integrals \eqref{singint} appear in the discretized Galerkin boundary equations in $I_{K_2}$, see \eqref{IK12}.

The considered scheme for (nearly) singular integrals incorporates a similar approximation technique to the one described in Section~\ref{Qregsplit}.  First, function $g$ is approximated by QI spline $\sigma_g$ on $D_i$. Then $B_i^{\cal H}\, \sigma_g$ is represented in B-form using B-spline basis $\{B^{(\btau_\Pi)}_{r,p+d}\}_r$ of degree $p+d$, defined on the local product space $\Pi$ of dimension $P$. Instead of definite integrals $\int_{D_i} B^{(\btau_\Pi)}_{r,p+d}$ in \ref{Qregsplit}, we need to compute the {\it modified moments}
\begin{align} \label{modmom}
\mu_r^{(i)}(s) := \int_{D_i} K_2(s,t)\, B^{(\btau_\Pi)}_{r,p+d}(t) \ dt\,, \qquad r=1,\ldots,P.
\end{align}

A recurrence formula to obtain exact expressions for the modified moments is derived from the B-spline recursive definition \cite{ACDS3}. The recurrence formula relies on given initial values, specifically on
$$
I_{K_2} (t^k\,\chi_{[c_1,c_2]},s) \,:=\, \int_{c_1}^{c_2}K_2(s,t) \,t^k\, dt \,, \qquad s \in [c_1, c_2],
$$
where $\chi_{[c_1,c_2]}$ is the characteristic function of the interval $[c_1, c_2].$
When dealing with an open boundary curve $\Gamma,$ it holds
\begin{align}\label{eqn:momentsValues}
\nonumber I_{K_2} (t^k\,\chi_{[c_1, c_2]},s) \, &=\, \sum_{j=0}^k \binom{k}{j} s^{k-j} \int_{c_1-s}^{c_2-s} \log|z| \cdot z^j\, dz\,\\
& =\, {\sum_{j=0}^k \binom{k}{j} s^{k-j}} \left. \frac{z^{k+1}}{k+1} \left( \log|z| - \frac{1}{k+1} \right) \right|_{c_1-s}^{ c_2-s}. 
\end{align}
We derive a similar  formula for the initial expressions when $\Gamma$ is closed. Recalling that in such case we have set $K_2(s,t) = \log(\delta(s,t))$ with
\begin{align}\label{eqn:newDistSplit}
\delta(s,t)={|s-t|\, |(s-t)^2-\gamma^2|}\,{\gamma^{-2}} = {|s-t|} \cdot {|(s-t)+\gamma|}\ {\gamma^{-1}} \cdot {|(s-t)-\gamma|}\,{\gamma^{-1}}
\end{align}
and hence $\log(\delta(s,t))$ is split into a sum of three functions. Modified moments for the function $|s-t|$ correspond to the ones already obtained for the case of the open boundary curve. The initial values of the recurrence formula for the latter two functions can be written in a compact form after some  simplifications of the derived expressions,
\begin{align}\label{eqn:momentsValues2}
\hspace{-.75cm}  \int_{c_1}^{c_2}\log\frac{|(s-t) \pm \gamma|}{\gamma} \cdot t^k\, dt \  \,=\, \sum_{j=0}^k \binom{k}{j} (\mp \gamma)^{j+1} (s \pm \gamma)^{k-j} \int_{\pm (s-c_1\pm \gamma)\, \gamma^{-1}}^{\pm (s-c_2 \pm \gamma)\,\gamma^{-1}} \log|z| \cdot z^j\, dz.
\end{align}
Analytical expressions for the emerged integrals in \eqref{eqn:momentsValues2} are obtained analogously as in \eqref{eqn:momentsValues}.
The  derived quadrature rule ${\cal I}^Q_{ {\it w}_i^s} [g]$ for $I_{ {\it w}_i^s} [g]$ can be compactly written as follows
\be \label{Q3}
{\cal I}^Q_{ {\it w}_i^s} [g] \,:=\,   {\bfeta^{(i)}}(s)  ^T\ \g^{(i)}\,,
\ee
where  $ \g^{(i)} := \left( g(\tau_0^{(i)}),\ldots, g(\tau_n^{(i)}) \right)^T$ and the weight vector $\bfeta^{(i)}(s)$ is defined as
\[
\bfeta^{(i)}(s) :=  \hat{C}^{(p)} \ G^{(p,d)}  \ \bmu^{(i)}(s) \ ,
\]
with $\bmu^{(i)}(s) := (\mu_1^{(i)}(s),\ldots,\mu_P^{(i)}(s))^T$.

Quadrature rule \eqref{Q3} is applied to approximate inner integrals in $I_{K_2}^{(i,j)}$, when they are singular or nearly singular.
For the outer integrals of $I_{K_2}^{(i,j)}$ we can instead use again (\ref{Q2}). Thus the quadrature for the double integral can be written using the following compact algebraic representation,
 \[
I_{K_2}^{(i,j)} \approx  {\w^{(i)}}^T\,  J^{(i)}\, M^{(i,j)}\, {G^{(p,d)}}^T\, {{}{\hat C}^{(p)}}^T\, J^{(j)}\, {\bf e},
\]
with ${\bf e} :=(1,\ldots,1)^T \in \RR^{n+1}$ and with $M^{(i,j)}$ denoting a matrix of size $(n+1) \times P$ with entries
\[
\left(M^{(i,j)} \right)_{k,r} := \mu_r^{(j)}(\tau_k^{(i)})\,.
\]
Clearly the introduced formula to approximate  $I_{K_2}^{(i,j)}$ is more involved than the one adopted for non-singular double integrals because it requires the preliminary computation of the modified moments in the matrix $M^{(i,j)}.$  Concerning its cost, it is worth to be mentioned that, if there exists a translation factor $\xi$ such that for another pair of indices $i'$ and $j'$ it is $B_{i'}^{\cal H}(\cdot) =  B_{i}^{\cal H}(\cdot -\xi)$ and $B_{j'}^{\cal H}(\cdot) =  B_{j}^{\cal H}(\cdot -\xi)$, then $M^{(i',j')} = M^{(i,j)}$ since the kernel $K_2$ depends only on the difference of  its arguments. Note that this is not uncommon in our uniform hierarchical setting.  A further reduction of the computational cost  can be obtained considering that, for the same reason, some of the entries of the matrix $M^{(i,j)}$ coincide \footnote{Even though the multiple knots are necessary to define the local product space \cite{Morken91}, it is also spanned by translates of few different B-splines because of the uniformity assumption on the initial space $S$ and of the choice of using dyadic hierarchical refinement.}.

\begin{rmk}
In \cite{ACDS3} all the inner integrals of $I_{K_2}^{(i,j)}$ were approximated by the same singular based quadrature rule, even if the kernel $K_2(s,t)$ did not locally exhibit any singularity. Relating to modified moments for Legendre polynomials, it was already observed in \cite{AD2002} that  their computation in finite arithmetic  can become increasingly unstable as the distance between $s$ and $D_i$ increases.  We observed a similar instability issue also in our experiments; for a fixed $s$ the instability increases also with higher spline degrees and smaller sizes of $D_i$. This motivates the use of regular based quadratures for $I_{K_2}^{(i,j)}$ when the inner integrals are regular. Also, a high-precision floating point arithmetic to evaluate the modified moments is advised.
\end{rmk}


\subsection{Accuracy of the quadrature rules for the boundary integrals} \label{AP}
\label{sec_quadaccuracy}

The matrix $A$ and the right-hand side vector $\bfbx$ are never computed exactly, due to quadrature errors. The theory of the perturbed Galerkin method guarantees that the optimal order of convergence of the perturbed Galerkin solution can be obtained if the size of the perturbation of $A$ and $\bfbx$ is sufficiently small \cite{BEMbook}. The amount of the perturbation can clearly be controlled by choosing a sufficiently accurate quadrature rule. 
Typically, the convergence properties of quadratures are studied with respect to the number of nodes. On the other hand, in order to design an efficient BEM scheme, a low amount of nodes is preferable. For that reason we are more interested to study the convergence of integrals with respect to the mesh size $h$, coming from the h--refinement, while maintaining the number of nodes fixed. 
Thus in this subsection  the accuracy of the formulas in (\ref{Q2}) and (\ref{Q3}) is studied with respect to $h_{\ell_i}$, which denotes the uniform size of the cells of level $\ell_i,$  where $1 \le \ell_i \le M,$ for each $i=1,\ldots,N_{\cal H}.$ More specifically, under suitable regularity assumption on the auxiliary factor $g$ we derive an upper bound for the  following two quadrature errors $e$ and $E$ in terms of the mesh size $h_{\ell_{i}}$, 
\begin{equation} \label{errquad}
e(h_{\ell_i}) \,:=\, I_{B_i^{\cal H}} [g] - {\cal I}^Q_{B_{i}^{\cal H}}[g]\,, \qquad  
E(h_{\ell_i},s) \,:=\, I_{w_i^s} [g] - {\cal I}^Q_{w_i^s}[g]\,.
\end{equation}
\begin{prn} \label{prn1}
Let  $g \in C^{p+1}[a, b]$.  Let the number $n+1$ of uniform nodes for the quadratures $ {\cal I}^Q_{B_{i}^{\cal H}}[g]$ and ${\cal I}^Q_{w_i^s}[g]$ be fixed. Then there exist two positive constants $C_1$ and $C_2$ not depending on the index $i$ such that
$$
\begin{array}{ll}
|e(h_{\ell_i})| & \leq   C_1  h_{\ell_i}^{{p+2}} \|D^{p+1} g \|_{L^{\infty}(\overline{D}_i)}, \cr
|E(h_{\ell_i},s)| & \leq C_2 h_{\ell_i}^{{p+2}}\,|\log h_{\ell_i}|\, \|D^{p+1} g \|_{L^{\infty}( \overline{D}_i)},	 
\end{array}
$$
where $h_{\ell_i}$, mesh spacing corresponding to $B_i^{\cal H}$, is sufficiently small.  
\end{prn}

\begin{proof}
The errors $e(h_{\ell_i})$ and $E(h_{\ell_i},s)$ are a result of the approximation step, where $g$ is approximated by the quasi--interpolant spline $\sigma_g$ on $D_i$.
This implies that
$$\begin{array}{ll} 
\left| e(h_{\ell_i}) \right| &=\,\left| \int_{D_i}  B_i^{\cal H}(t) [g(t) - \sigma_g(t)] dt \right| \, \leq \, \|g-\sigma_g\|_{L^\infty(\overline{D}_i)} \int_{D_i}  B_i^{\cal H}(t)  dt,\\[2ex]
\left| E(h_{\ell_i},s)\right| &= \, \left| \int_{D_i}  \log (\delta(s,t))\, B_i^{\cal H}(t) [g(t) - \sigma_g(t)] dt \right| \\ &\leq \,
  \|g-\sigma_g\|_{L^\infty(\overline{D}_i)} \int_{D_i}  |\log(\delta(s,t)) | B_i^{\cal H}(t)  dt \,.
\end{array}
$$
The convergence property of the quasi--interpolation scheme that was proven in \cite{MSbit09} ensures that there exists a constant $L_0$ such that
\begin{align*}
\hspace{-.75cm} \|g-\sigma_g\|_{L^\infty(\overline{D}_i)}  \leq  L_0 | \btau^{(i)} |^{p+1} \, \|D^{p+1} g\|_{L^\infty(\overline{D}_i)} \,\leq\,  L_0\, \left(\frac{ d+1}{n}\right)^{p+1}\,  h_{\ell_i} ^{{p+1}} \, \|D^{p+1} g\|_{L^\infty(\overline{D}_i)}.
 \end{align*}
 The second inequality holds true, since, due to the uniformity of $\btau^{(i)}$ on $D_i$, $n |\btau^{(i)}| \leq (d+1) h_{\ell_i}$.

By applying (\ref{eqn:integBspline}) we immediately get
$$
\int_{D_i}  B_i^{\cal H}(t)  dt \, \,\le h_{\ell_i}
$$ 
and the inequality for $|e(h_{\ell_i})|$ is proven by setting $C_1 := L_0\, ( d+1)^{p+1} n^{-p-1}$.

To prove the inequality for $|E(h_{\ell_i},s)|$, 
the integral in the estimate can be bounded by
\begin{align}\label{eqn:integIneq}
\int_{D_i} | \log \delta(s,t)| B_i^{\cal H}(t)  dt \leq \int_{D_i} | \log \delta(s,t)|\, dt.
\end{align}

First, let us focus on the case $\delta(s,t) = |s-t|$ and let us assume that $D_i \cap [s-e^{-1},s+e^{-1}] = \emptyset \,,$  where $e$ is the Euler's number. Then we can bound the integrand $|\log|s-t||$ in \eqref{eqn:integIneq} by a constant $C'_2:=\max\{1, \log (b-a) \}.$ Hence the integral can be further bounded as follows,
\begin{align*}
\int_{D_i} | \log |s - t|| dt \leq \int_{D_i} C'_2\, dt = C'_2\, (d+1)\, h_{\ell_i}
\end{align*}
and the estimate for $|E(h_{\ell_i},s)|$ holds true for $C_2 := L_0\, ( d+1)^{p+2} n^{-p-1} C'_2$.

Now let us consider the case $D_i \subset [s-e^{-1},s+e^{-1}].$  Setting  $z:=s-t$, since $\int \log |z|\, dz \,=\, z\,(\log |z| -1)$ and the function $z\,|\log z|$ is  monotonically increasing in $(0, e^{-1})$, we get
\begin{align*}
\int_{D_i} |\log|s-t||\, dt &= z\, (\log |z| -1)\big|_{s-\tau_0^{(i)}}^{s-\tau_n^{(i)}}\\
&= (d+1) h_{\ell_i} - (s-\tau_n^{(i)})\, \big|\log | s - \tau_n^{(i)}|\big| +  (s-\tau_0^{(i)})\, \big|\log |s-\tau_0^{(i)}|\big|\\
& \leq (1+2 \log (d+1))\, (d+1)\, h_{\ell_i} + 2 (d+1)\, h_{\ell_i}\, | \log h_{\ell_i}|.
\end{align*}
Therefore we can find a constant $C'_2>0\,,$ depending only on $d,$ such that
\begin{align*}
\int_{D_i} | \log |s -t|| dt   \leq C'_2\,  h_{\ell_i}\, | \log h_{\ell_i} |.
\end{align*}
Hence we set $C_2 := L_0\, ( d+1)^{p+1} n^{-p-1} C'_2$. \\
The remaining case when $D_i$ is only partially inside $[s-e^{-1},s+e^{-1}]$ is similar to the previous ones if the integration domain is splitted into two intervals.

 When  $\delta(s,t)={|s-t|}\,{|(s-t)^2-\gamma^2|}\,{\gamma^{-2}}$, we can split the estimate \eqref{eqn:integIneq} into three parts using the relation \eqref{eqn:newDistSplit},
\begin{align*}
\int_{D_i} | \log \delta(s,t)|\, dt \leq 
\int_{D_i} |\log{|s-t|}| + \left|\log \frac{|(s-t)+\gamma|}{\gamma}\right| + \left|\log \frac{|(s-t)-\gamma|}{\gamma}\right| \, dt.
\end{align*}
Upper bounds for the latter two integrand functions are obtained in a similar way as for the first one by considering shifts and scaling by $\gamma$ of the logarithmic function.

\end{proof}

\begin{rmk}
As well as for the original quadrature rule reported in (\ref{Q1}), we have experimented that, when $p$  is even, the approximation power of both the rules introduced in (\ref{Q2}) and (\ref{Q3}), implemented using uniform nodes, increases of one order.
\end{rmk}

In the remainder of the section we present the results of two numerical experiments aimed  to corroborate Proposition~\ref{prn1} for the estimates $e$ and $E$ in \eqref{errquad}.  Convergence of the error with respect to $h_{\ell_i}$ is studied on a sequence of uniform meshes. The errors are computed by using as exact integral values those  obtained with the integration solver in Wolfram Mathematica.

Let $d=p=2$ and let us consider the function $g(t)=\sqrt{1+4t^2}$. In both the experiments we use two different values for the number of quadrature nodes $n+1$:   $n = 5$ and $n= 25$. Then, in the first experiment, see the picture on the left of Fig. \ref{fig:singIntConvergence}, we compute the quantity $ \max_i  | e(h_{\ell_i} ) | $ measured for successively halved values of $h_{\ell_i}$. The quadratic B-splines $B_{i}^{\cal H}$ are constructed on the integration interval $[-1, 1]$, by using uniform open knot vectors for the following mesh sizes: $h_{\ell_i}=1/5 \cdot 2^{-\ell}$ for $\ell = 0,1,2,3$. 
Note that $g$ is regular in the integration domain and actually the error behaviour shown in the figure exhibits order of convergence $O(h^5_{\ell_i})$.
As a second test, relating to the same $g$ function, we analyze the accuracy of the derived QI quadrature scheme for the integrals \eqref{singint} by computing the values $E(h_{\ell_i},s)$ for all the B-splines and varying  the parameter $s$ in the set of all the spline breakpoints and their midpoints.  
 The convergence of the error $\max_{i,s} |E(h_{\ell_i},s)|$ shown on the right of Figure~\ref{fig:singIntConvergence} reveals the convergence order $O(h_{\ell_i}^5 |\log h_{\ell_i}|)$.}  
Note that, as expected,  in both the experiments the accuracy of the quadrature is  improved when the number $n+1$ of quadrature nodes is increased.
\begin{figure}[t!]
\centering 
\includegraphics[trim = 0cm 0cm 0.5cm 0.25cm, clip = true, height=5cm]{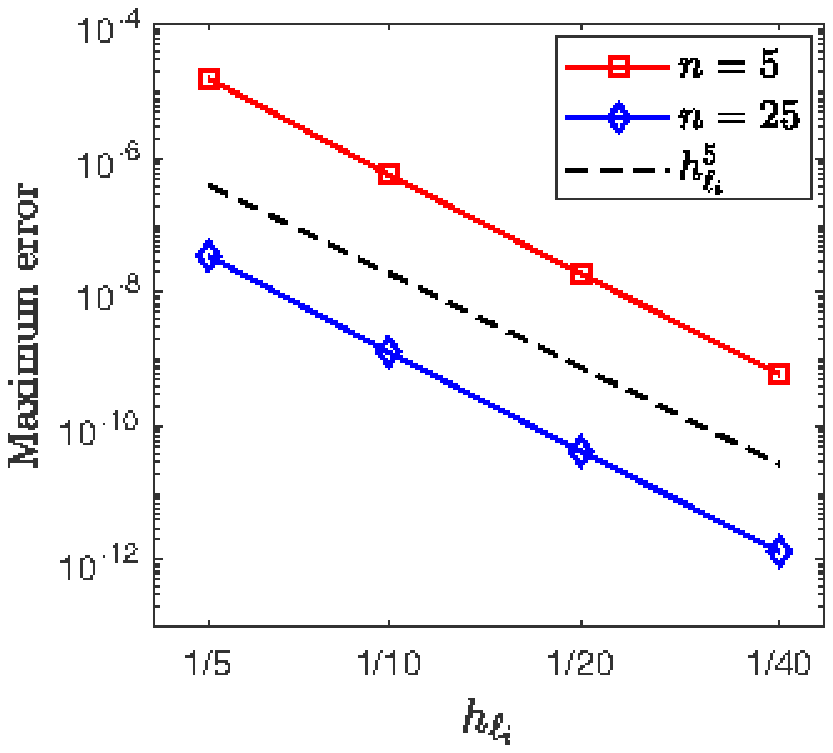}
\hspace{1cm}
\includegraphics[trim = 0cm 0cm 0.5cm 0.25cm, clip = true, height=5cm]{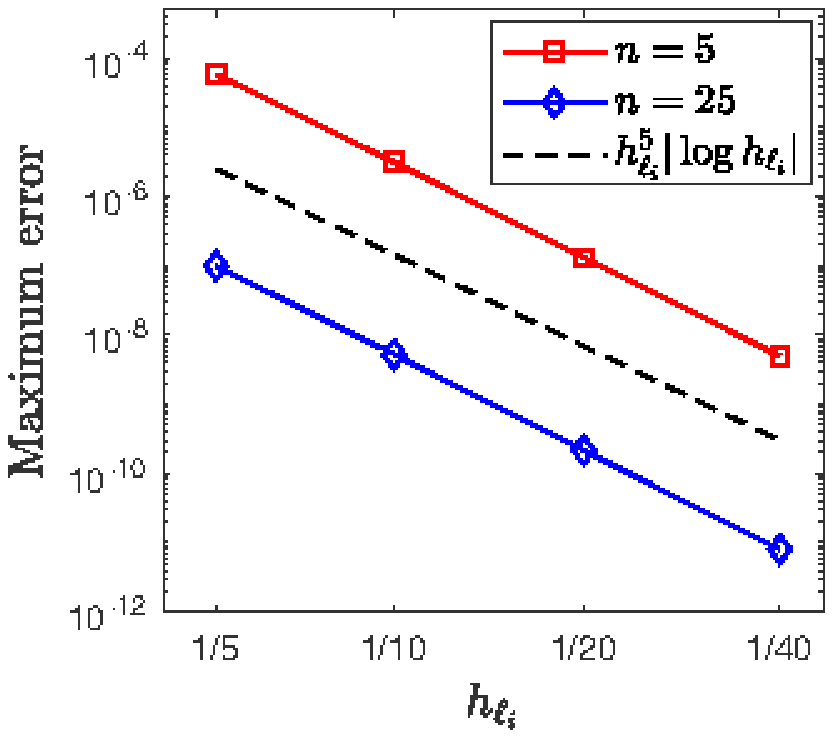}
\caption{Maximum errors of $|e(h_{\ell_i})|$ (left) and of $|E(h_{\ell_i},s)|$ (right) with respect to the mesh sizes $h_{\ell_i}$ and $d=p=2$.}
\label{fig:singIntConvergence}
\end{figure}
A more detailed analysis of  the features of the derived QI based quadrature schemes is given in a forthcoming paper\footnote{A. Falini, T. Kandu\v{c}. A study on spline quasi-interpolation based quadrature rules for the isogeometric Galerkin BEM, \emph{in preparation}.}, 
where also a comparison with some other quadrature methods is done.

\section{Numerical examples}\label{sec:num}

In this section we test our model to numerically solve three Laplace boundary value problems. 
For all examples the initial spaces and meshes are constructed on uniformly spaced meshes.
A local dyadic refinement procedure is steered by the residual based error estimator and D\"{o}rfler marking strategy.

\subsection{Crack problem on a slit}

In this example we focus on a crack problem reported in \cite{feischl2015reliable}.
The slit in Figure~\ref{fig:crack}(a) is defined as $\Gamma = [-1,1]\times \left\{0\right\}$. For $\x=(x_1,x_2)\in\RR^2$ and the right-hand side $f(x_1,0) = -x_1/2$, the exact solution of the Symm's equation is equal to
$\phi(x_1,0) =  -{x_1}\, {({1-x_1^2})^{-1/2}}$, which has poles at $x_1=-1,1$ (see Figure~\ref{fig:crack}(d)). The same color gradient is used in Figure~\ref{fig:crack}(a) and Figure~\ref{fig:crack}(d) to match the corresponding points in the physical and the parametric domain. Since $\phi\in H^{-1/2}(\Gamma) \backslash L^2(\Gamma)$ we measure the error of the approximated solution in the energy norm $\vertiii{\bullet}$ induced by the elliptic operator $V$,
$$
\vertiii{\phi}^2:= \langle V\phi,\phi\rangle_{L^2(\Gamma)}.
$$
Orthogonality of the approximated Galerkin solution with respect to the exact one allows us to compute the error by the formula $\vertiii{\phi - \phi_h}= \sqrt{\vertiii{\phi}^2 - \vertiii{\phi_h}^2}$ and it can be easily checked that $\vertiii{\phi}^2 = \pi/4$. Energy norm of the approximated solution is computed by the formula $\vertiii{\phi_h}^2 = \langle V\phi_h,\phi_h \rangle_{L^2(\Gamma)} = \bfax^T \bfbx_1$.

\begin{figure}[t!]
\centering
\subfigure[Geometry]{
\includegraphics[trim = 0.6cm 0cm .75cm 0cm, clip = true, height=2.35cm]{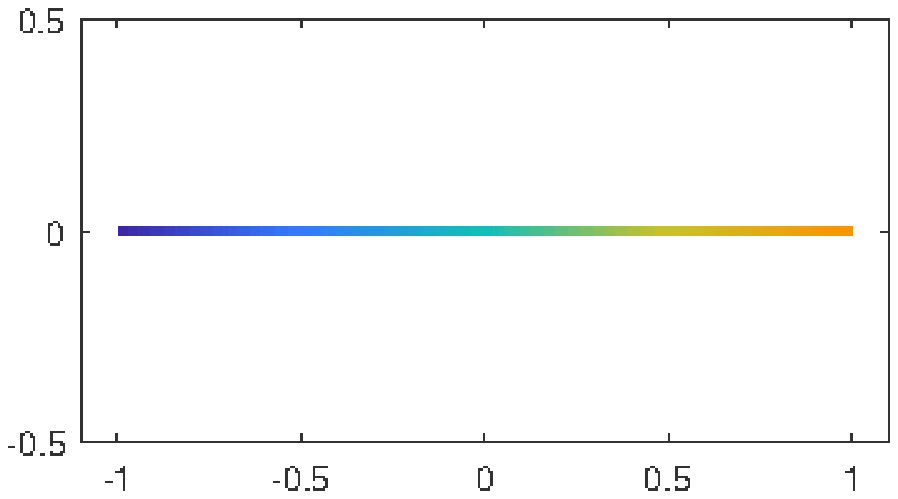}
}
\subfigure[Initial mesh]{
\includegraphics[trim = 0.6cm 0cm .75cm 0cm, clip = true, height=2.35cm]{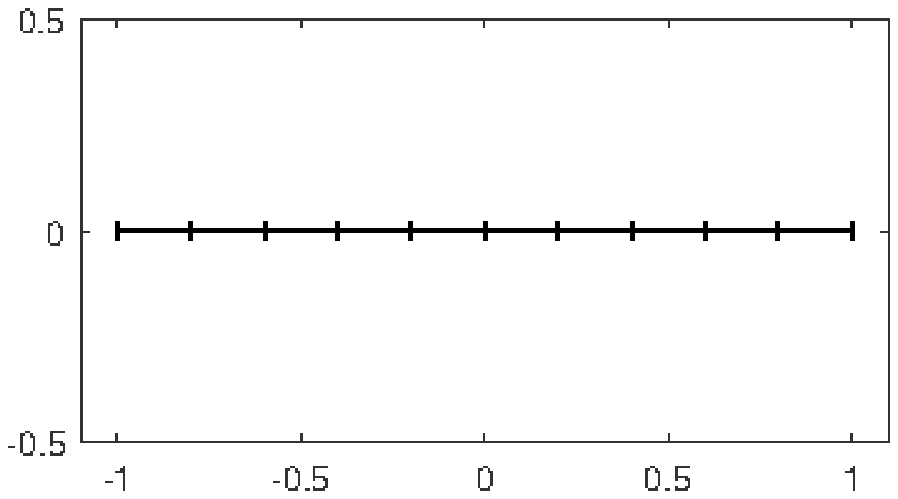}
}
\subfigure[Hierarchical mesh ($10^{\rm th}$ iteration)]{
\includegraphics[trim =0.6cm 0cm .75cm 0cm, clip = true, height=2.35cm]{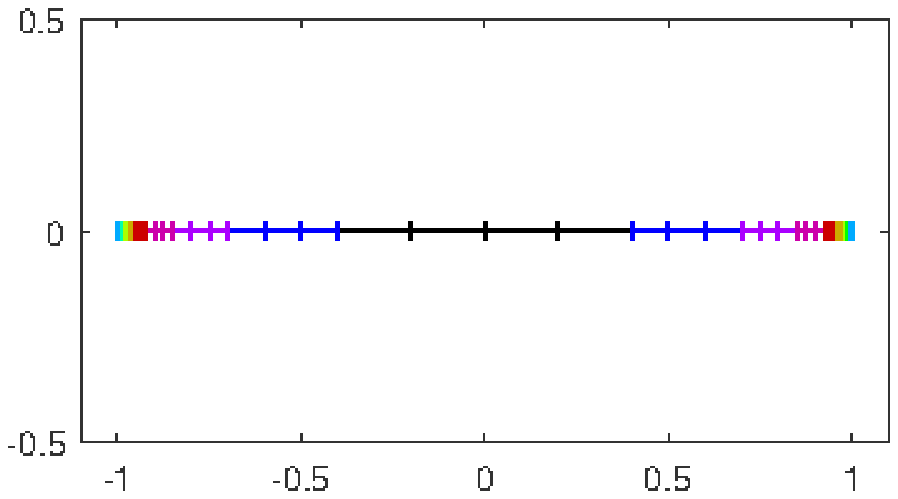}
}
\subfigure[Exact solution]{
\includegraphics[trim = 0cm 0cm .5cm 0.25cm, clip = true, height=4.5cm]{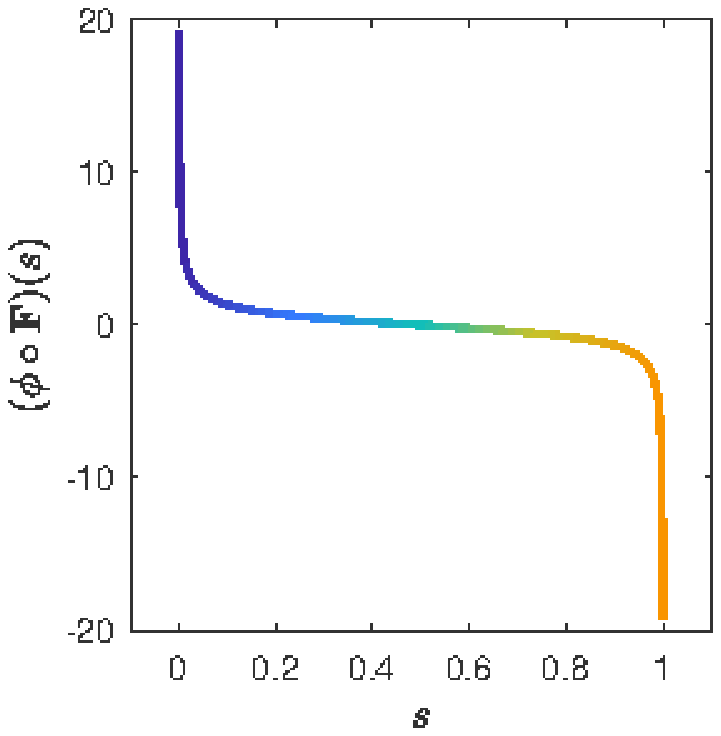}
}
\hspace{1cm}
\subfigure[Convergence]{
\includegraphics[trim = 0cm 0cm .5cm 0.25cm, clip = true, height=4.5cm]{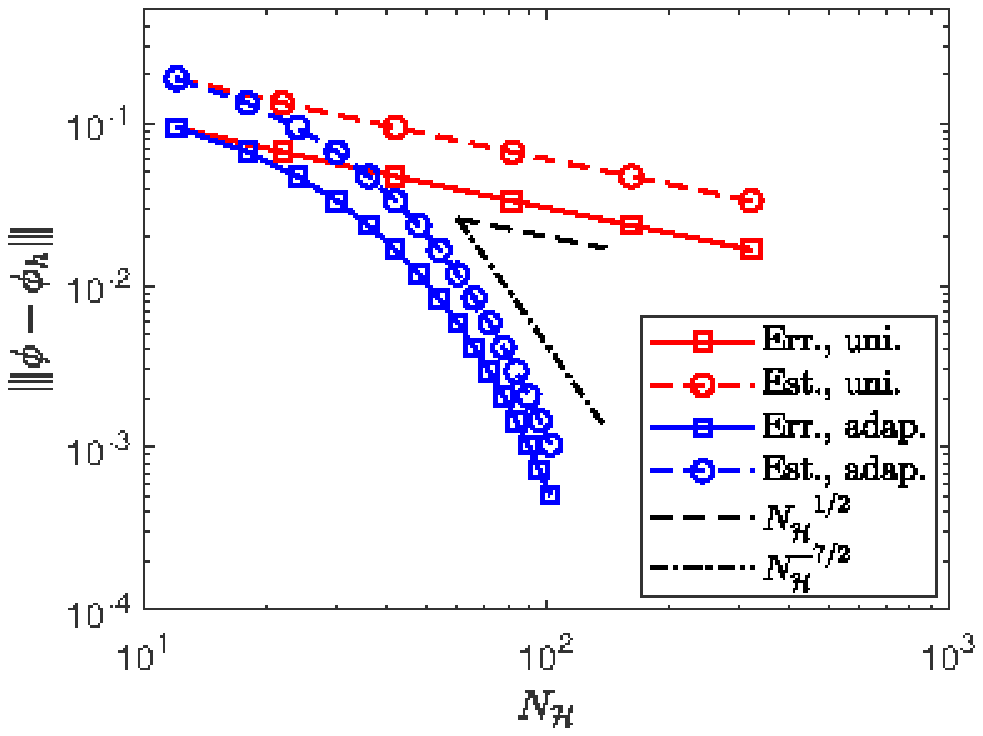}
}
\caption{Crack problem on a slit.}
\label{fig:crack}
\end{figure}

The approximated solution of the problem is sought in the space of quadratic hierarchical splines ($d=2$) with the uniform extended knot vector $
T  = \Big( 0, 0, 0, 1/5,\dots, 4/5, 1,\\ 1, 1 \Big)$, and the geometry $\Gamma$ defined by the control points 
{
\begin{align*}
D &=
\begin{bmatrix}
-1&	- 4/5&	-2/5&	0&		2/5&	4/5&	1\\
0&		0&		0&		0&	0&		0&		0
\end{bmatrix}.
\end{align*}
}
The initial mesh is depicted in Figure~\ref{fig:crack}(b).

Since the geometry is parametrized by a linear function, the parametric speed is constant, $J \equiv {2}$. Hence the inner integrals \eqref{IK12} of the system matrix $A$ are just the modified moments \eqref{modmom}, described in Section~\ref{Qsingsplit}. Thus to solve them accurately it is enough to take a small number of quadrature nodes, for example $n = 6$, 
and to use the quadratic quasi--interpolant splines ($p=2$). A higher value, $n=12$, is necessary for the outer integrals in order to achieve a steady convergence.

In the first test we compare the exact solution with the approximated ones obtained by applying several uniform global refinements. The error of the approximation solution and the corresponding values of the error estimator with respect to the degrees of freedom $N_{\cal H}$ are plotted in Figure~\ref{fig:crack}(e). Due to reduced regularity of the exact solution the expected convergence order of the approximated solution is $-1/2$, which is confirmed by our experiment. Figure~\ref{fig:crack}(e) also shows that the convergence order is greatly improved by applying the adaptive scheme with the local refinement strategy and D\"{o}rfler marking parameter set to $\theta = 1-10^{-2}$. The theoretical optimal convergence order for regular enough solutions, $-7/2$, is recovered after a few refinement steps. At every iteration step, only a few cells at the highest active hierarchical level and near the singularities are refined. The hierarchical mesh is shown in Figure~\ref{fig:crack}(c). Mesh cells of different levels are coloured differently to improve the visibility of the hierarchical mesh. 



\subsection{Pac-Man-like domain}

In the second example we consider a closed domain problem by using the direct approach integral equations \eqref{seconda_BIE}. The boundary $\Gamma$ of the domain is a smooth B-spline circular sector, sometimes referred to as the Pac-Man geometry, see Figure~\ref{fig:smoothPacman}(a). It is described by cubic B-splines on the uniform extended knot vector 
$\mathbf T = (-9/6,\ -8/6, \ldots,
8/6,\ 9/6)
$
and the following control points:
\begin{align*}
D &=
\large \left[
\begin{smallmatrix}
-1&	-1/3&	2/5&	7/8&	7/8&	-1/25&	-1/25&	7/8&	7/8&	2/5& 	-1/3&	-1&	-1&	-1/3&	2/5\\
-1/2&	-1&	-1&	-1/2&	-1/2&	0&   0&		1/2&	1/2&	1&		1&		1/2&	-1/2&		-1&	-1
\end{smallmatrix}\right].
\end{align*}
The boundary $\Gamma$ is described by $s \in [-1,1]$ in the parametric space. The initial mesh is shown in Figure~\ref{fig:smoothPacman}(b).

Following the construction from \cite{feischl2015reliable} the exact solution of the Laplace equation is set to
\begin{align*}
u(r,\vartheta) =  -r^{1/2} \cos \frac {\vartheta + \pi} 2,
\end{align*}
written in polar coordinates, with $x_1 =r \cos \vartheta, x_2 = r \sin \vartheta$, for $r>0$ and $\vartheta\in (0,2\pi)$. 
The exact solution of the integral equation is equal to 
\begin{align*}
\phi(r,\vartheta) = \frac 1 2 r^{-1/2}
\begin{bmatrix}
\displaystyle -\sin \frac \vartheta 2 & \displaystyle \cos \frac \vartheta 2
\end{bmatrix}
\boldsymbol n(r,\vartheta),
\end{align*}
which exhibits a singular point at the origin $r=0$. Note that the singular point lies outside our domain $\Omega$ and so we can use the standard $L^2$ norm to measure the error. Nevertheless, the exact solution $\phi$ has a strong feature near the value $s = -1/4$ in the parametric space, as seen in Figure~\ref{fig:smoothPacman}(d).

The mentioned feature prevents the global uniform refinement strategy to recover the optimal convergence order in the first refinement steps (Figure~\ref{fig:smoothPacman}(e)). Due to sharper corners of the geometry, we need to set $n=36$ 
for the outer quadrature scheme, and $n = 12$ 
for the inner, and it is enough to employ the quadratic QI splines. The optimal convergence order 4 of the $L^2$ error is recovered, when we employ the adaptive refinement scheme with the threshold parameter $\theta = 4/5$. The error estimator and the marking strategy correctly steer the mesh refinement near the three corners, as depicted in Figure~\ref{fig:smoothPacman}(c).

\begin{figure}[t!]
\centering
\subfigure[Geometry]{
\includegraphics[trim = 0cm 0cm .5cm 0cm, clip = true, height=4.25cm]{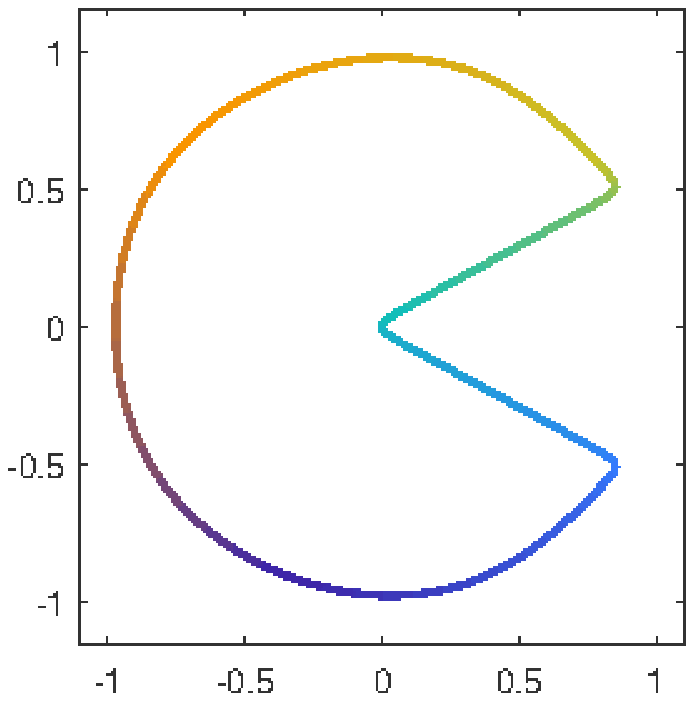}
}
\subfigure[Initial mesh]{
\includegraphics[trim = 0cm 0cm .5cm 0cm, clip = true, height=4.25cm]{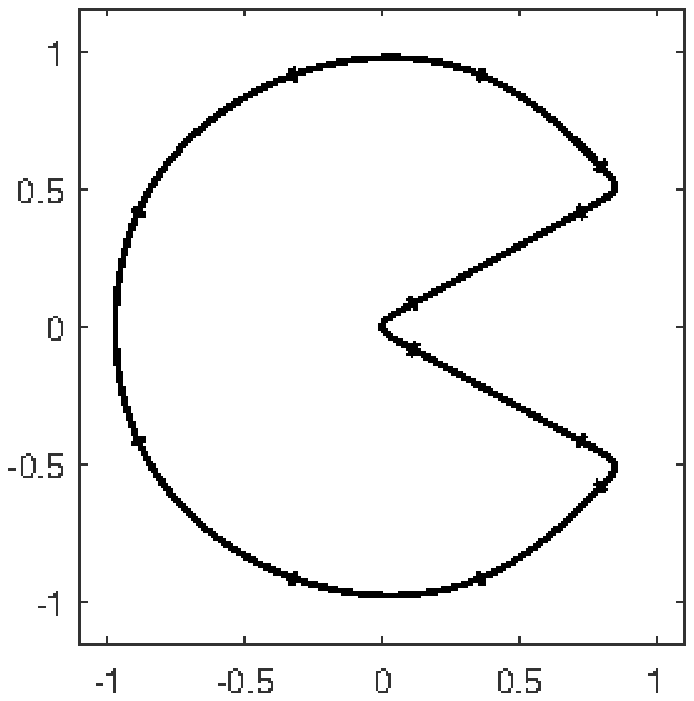}
}
\subfigure[Hierarchical mesh ($6^{\rm th}$ iteration)]{
\includegraphics[trim = 0cm 0cm .5cm 0cm, clip = true, height=4.25cm]{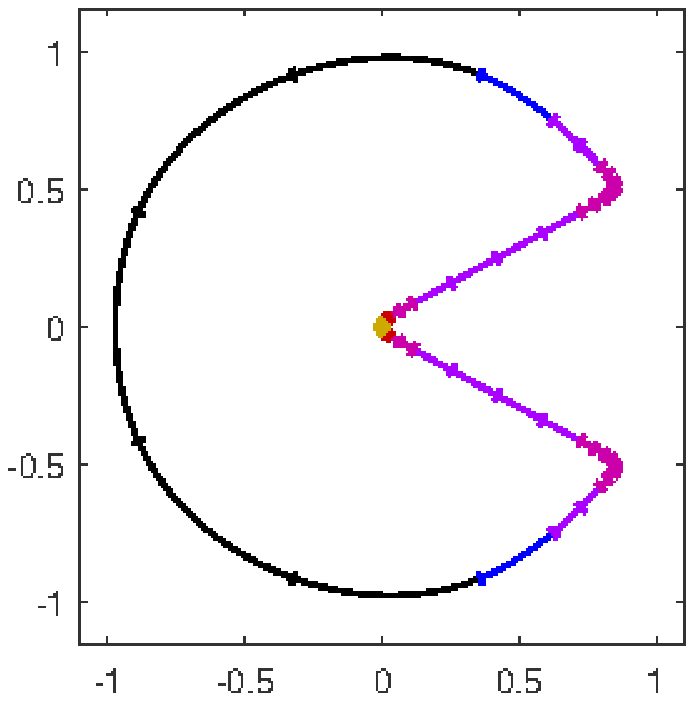}
}
\subfigure[Exact solution]{
\includegraphics[trim = 0cm 0cm .5cm 0.25cm, clip = true, height=4.5cm]{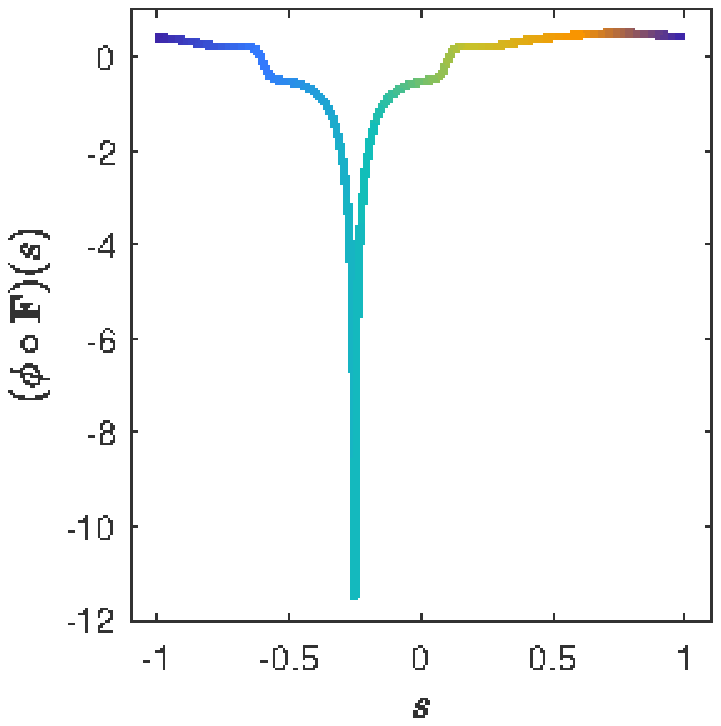}
}
\hspace{1cm}
\subfigure[Convergence]{
\includegraphics[trim = 0cm 0cm .5cm 0.25cm, clip = true, height=4.5cm]{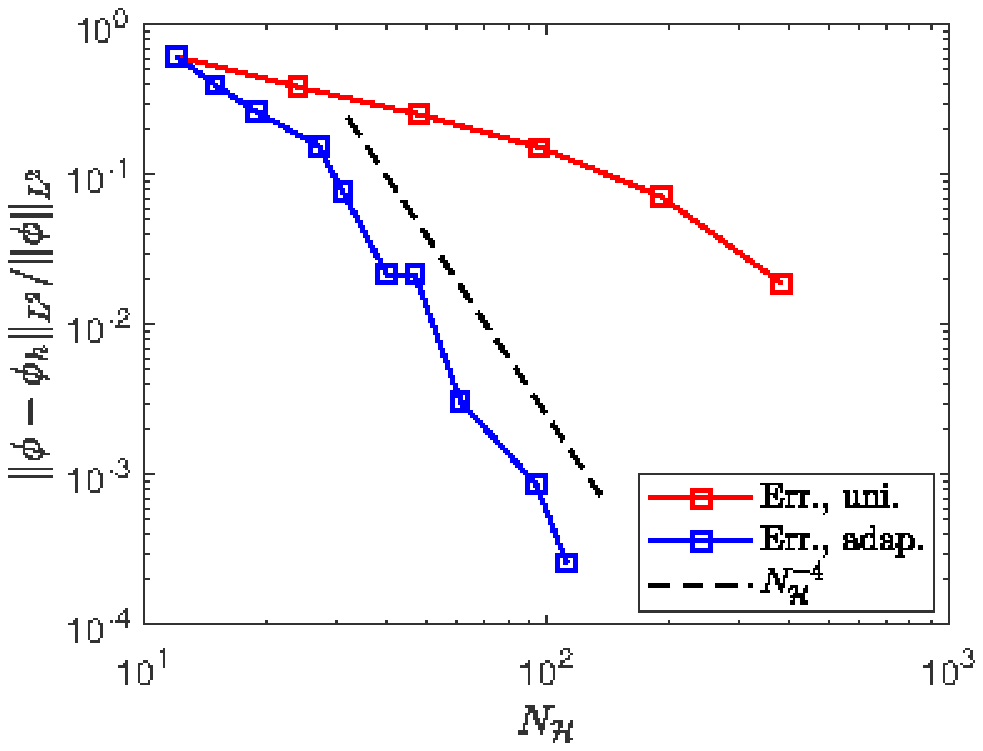}
}
\caption{Pac-Man-like domain problem.}
\label{fig:smoothPacman}
\end{figure}


\subsection{L-shaped domain}

In the last example we study a smooth L-shaped domain in Figure~\ref{fig:Ldomain}(a). It is parametrized by a cubic B-spline curve, defined on the parametric interval $[-1,1]$, uniform extended knot vector $\mathbf T = (-13/10,\ -12/10,\ldots
\ -1/10,\ 0)$, 
and control points
{
\begin{align*}
%
%
D &= 
\Large \left[
\begin{smallmatrix}
0&		0&		0&		0&		-\varepsilon&	-\bar \varepsilon&	-1&		-1&		-1&		-1&		-1&		 -\bar \varepsilon&	0&		\bar \varepsilon&	1&		1&		1&		1&		\bar \varepsilon&	\varepsilon&	0&		0& 	0\\
0&		\varepsilon&	\bar \varepsilon&	1&		1&		1&		1&		\bar \varepsilon&	0&		-\bar \varepsilon&	-1&	-1&	-1&	-1&	-1&	-\bar \varepsilon&		-\varepsilon&		0&		0&		0&		0&		\varepsilon&		\bar \varepsilon
\end{smallmatrix}\right],
\end{align*}
}%
for $\varepsilon = 1/50$ and $\bar \varepsilon = 49/50$. The initial mesh is presented in Figure~\ref{fig:Ldomain}(b).

For $\boldsymbol x \in \RR^2 \backslash \{\boldsymbol \delta \}$ the exact solution $u$ of the Laplace equation is set to
\begin{align*}
 u(\boldsymbol x) = \frac 1 2 \log \|\boldsymbol x+\boldsymbol \delta\|_2^2.
\end{align*}
The exact solution of the Symm's equation then reads
\begin{align*}
\phi(\boldsymbol x) = \frac{(\boldsymbol x+\boldsymbol \delta)^T\,  \boldsymbol n(\boldsymbol x)}{\|\boldsymbol x+\boldsymbol \delta\|_2^2}.
\end{align*}
We set  $\boldsymbol \delta= -1/250\  (1, 1)^T$ so that  the singular point $\boldsymbol x = \boldsymbol \delta$ is outside the domain $\Omega$ and we can measure the error of the approximated solution in $L^2$ norm. As seen in Figure~\ref{fig:Ldomain}(d), the values of the exact solution $\phi$ decrease rapidly near $s = 9/10$; a region where the mesh of the approximated solution should be refined.

\begin{figure}[t!]
\centering
\subfigure[Geometry]{
\includegraphics[trim = 0cm 0cm .5cm 0cm, clip = true, height=4.25cm]{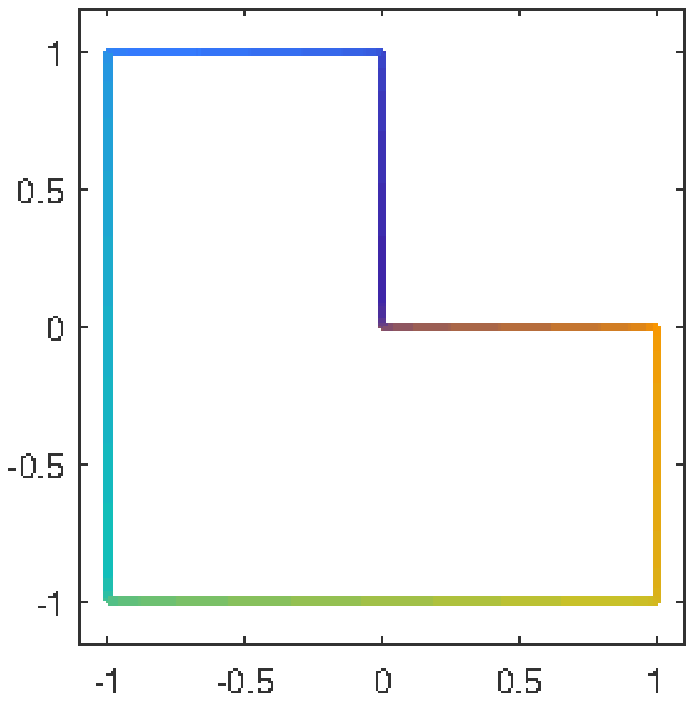}
}
\subfigure[Initial mesh]{
\includegraphics[trim = 0cm 0cm .5cm 0cm, clip = true, height=4.25cm]{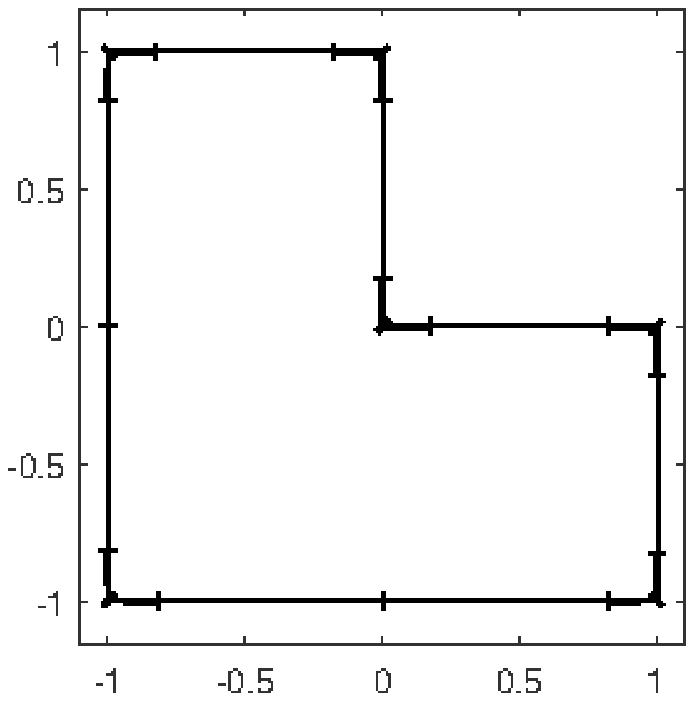}
}
\subfigure[Hierarchical mesh ($6^{\rm th}$ iteration)]{
\includegraphics[trim = 0cm 0cm .5cm 0cm, clip = true, height=4.25cm]{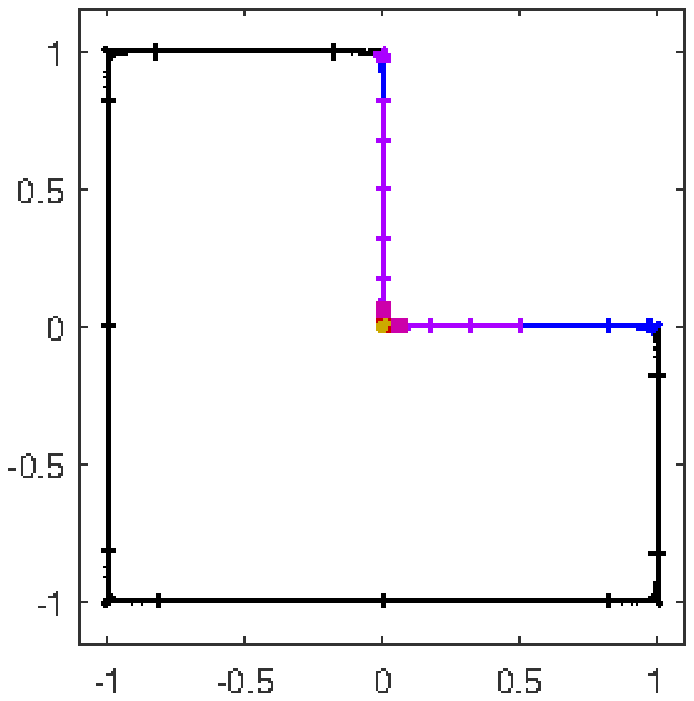}
}
\subfigure[Exact solution]{
\includegraphics[trim = 0cm 0cm .5cm 0.25cm, clip = true, height=4.5cm]{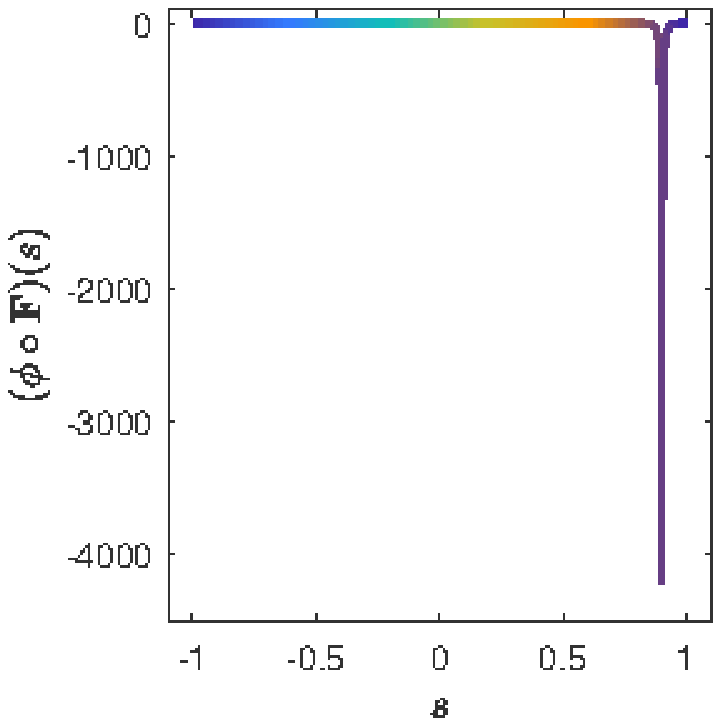}
}
\hspace{1cm}
\subfigure[Convergence]{
\includegraphics[trim = 0cm 0cm .5cm 0.25cm, clip = true, height=4.5cm]{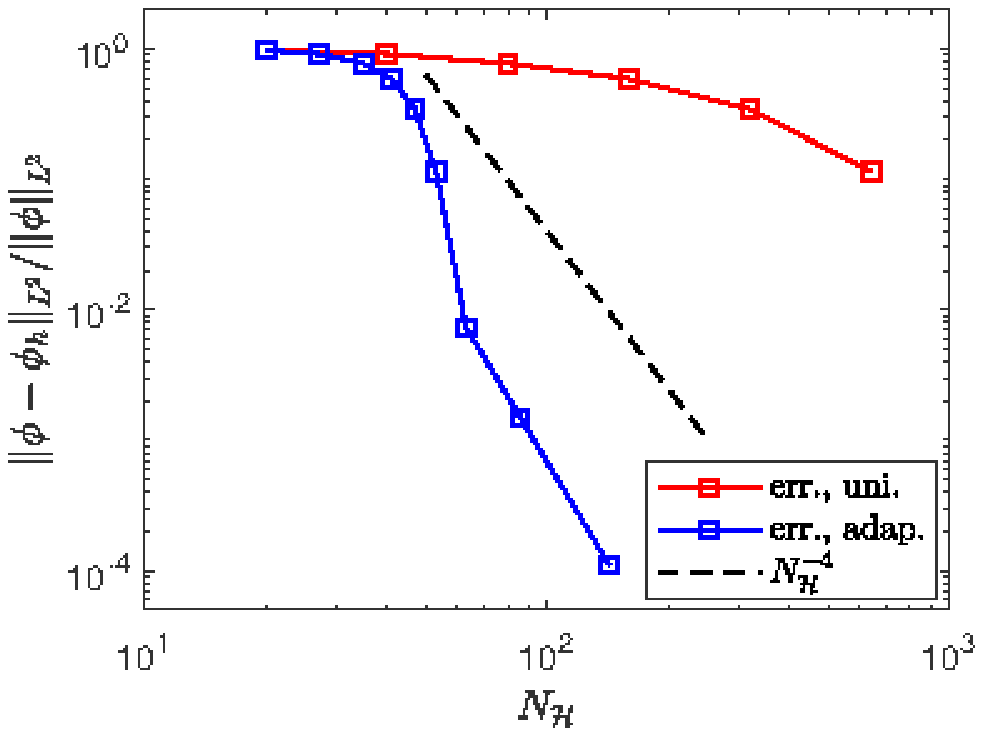}
}
\caption{The L-shaped domain problem.}
\label{fig:Ldomain}
\end{figure}

We set the quadrature parameters to $n=12$ 
for both the inner and the outer rule and we choose the quadratic QI splines. In Figure~\ref{fig:Ldomain}(e) we can observe that with the global uniform refinement strategy we get the sub optimal convergence. Instead, by setting the marking parameter to $\theta=99/100$ we recover the optimal convergence order with the adaptive scheme after a few iterations. The mesh is refined mainly near the corner $(0, 0)^T$, as depicted in Figure~\ref{fig:Ldomain}(c).


\section{Conclusion} \label{sec:conc}
We developed an adaptive IgA-BEM with hierarchical B-splines and high order quadrature schemes based on spline quasi--interpolation. Such kind of formulas do not require to split the considered integral into a sums of integrals on mesh cells, as it is commonly done when dealing with the Lagrangian basis.
Moreover, as the quadrature schemes are tailored on B-spines, the nodes can be taken on the support of each basis function. By using uniform knot sequences at any hierarchical level, the computation of the quadrature rules is also highly simplified. Implementation of the hierarchical structure leads to an effective adaptive IgA--BEM model.
The numerical results confirm that the local nature of the new quadrature rules, based on quasi-interpolation, perfectly fits within the adaptive hierarchical spline framework.

\section*{Acknowledgements}
This work was partially supported by the MIUR “Futuro in Ricerca” programme through the project DREAMS (RBFR13FBI3).
The authors are members of the INdAM Research group GNCS. The INdAM support through GNCS and Finanziamenti Premiali SUNRISE is gratefully acknowledged.

\section*{References}


\bibliography{BEMquad}

\end{document}